\documentclass[11pt]{amsart}

\usepackage{amsfonts,amsmath,amssymb,amscd,amsthm}
\usepackage{mathrsfs}
\usepackage{faktor}    
\usepackage[utf8]{inputenc}
\usepackage[english]{babel}
\usepackage{bbm}
\usepackage{pst-node}
\usepackage{enumitem}
\usepackage{tikz-cd}
\usepackage{tikz}
\usetikzlibrary{arrows, automata, positioning}
\usepackage{pdfpages}
\usepackage[left=37mm, top=20mm, right=37mm, bottom=25mm]{geometry} 
\usepackage{tabto}
\TabPositions{2 cm, 4 cm, 6 cm, 8 cm}
\usepackage[unicode, pdftex]{hyperref}
\hypersetup{
  colorlinks   = true, 
  urlcolor     = blue, 
  linkcolor    = blue, 
  citecolor   = red 
}

\DeclareMathOperator{\spam}{span}

\DeclareMathOperator{\Ra}{Re}

\DeclareMathOperator{\tr}{tr}
\DeclareMathOperator{\op}{op}

\DeclareMathOperator{\supp}{supp}

\newcommand{\Ch}[1]{\mathrm{Ch}\left(#1\right)}
\newcommand{\Tr}[1]{\mathrm{Tr}\left(#1\right)}

\DeclareMathOperator{\Sym}{Sym}

\DeclareMathOperator{\SL}{SL}

\newcommand{\ultprod}[1]{\prod_{#1 \to \omega}}
\newcommand{\ultlim}[1]{\lim_{#1 \to \omega}}

\newcommand{\T}{{\mathbb{T}}}
\newcommand{\R}{{\mathbb{R}}}
\newcommand{\F}{{\mathbb{F}}}

\newcommand{\N}{{\mathbb{N}}}
\newcommand{\Z}{{\mathbb{Z}}}
\newcommand{\C}{{\mathbb{C}}}
\newcommand{\cH}{{\mathcal{H}}}

\newcommand{\cQ}{{\mathcal{Q}}}
\newcommand{\cP}{{\mathcal{P}}}

\newcommand{\cN}{{\mathcal{N}}}

\newcommand{\cI}{{\mathcal{I}}}

\newcommand{\fc}{\mathfrak{c}}
\newcommand{\bM}{{\mathbf{M}}}
\newcommand{\bE}{{\mathbf{E}}}
\newcommand{\bG}{{\mathbf{G}}}

\newcommand{\bH}{{\mathbf{H}}}

\newtheorem{theorem}{Theorem}[section]
\newtheorem{corollary}[theorem]{Corollary}
\newtheorem{lemma}[theorem]{Lemma}

\newtheorem{proposition}[theorem]{Proposition}

\newtheorem*{lemma*}{Lemma}
\newtheorem*{proposition*}{Proposition}
\newtheorem*{theorem*}{Theorem}
\newtheorem*{corollary*}{Corollary}
\newtheorem*{claim*}{Claim}

\theoremstyle{definition}
\newtheorem{definition}[theorem]{Definition}

\newtheorem*{definition*}{Definition}

\theoremstyle{remark}
\newtheorem{remark}[theorem]{Remark}

\theoremstyle{definition}

\newtheorem{question}[theorem]{Question}

\title[hyperlinearity, stability and spectral gap]{Hyperlinearity, stability and asymptotic spectral gap of higher rank lattices}
\author{Alon Dogon and Itamar Vigdorovich}


\begin{document}

\begin{abstract}
We prove that if the group $\SL_2(\mathbb Z[1/p])$ is flexibly Hilbert--Schmidt stable for some prime $p$, then it admits a non-hyperlinear finite central extension.
Consequently, a positive answer to the following question would yield an explicit example of a non-hyperlinear group: If two  representations of the modular group $\mathrm{SL}_2(\mathbb{Z})$ almost agree on an Iwahori subgroup $B$, 
must they be close to representations that agree on $B$?
More generally, we investigate spectral gap properties for asymptotic representations of higher rank lattices and groups with property (T:FD).
In this setting, we prove that character rigidity is equivalent to a weak form of stability.

\end{abstract}

\maketitle

\section{Introduction}   \label{sec:intro}

One of the most notorious open problems in group theory concerns the existence of non-sofic and non-hyperlinear groups.
The latter can be defined as follows:
Recall that the \emph{normalized Hilbert--Schmidt norm} of a matrix $x \in M_m(\C)$ is $\| x \|_2 = \| x\|_{2,m} :=  (\frac{1}{m} \tr(x^*x))^{1/2}$.

\begin{definition}\label{def:hyperlinear}
Let $\Gamma$ be a group. An \emph{asymptotic representation} is sequence of maps \( \pi_n\colon\Gamma\longrightarrow U(d_n) \) satisfying
$$\lim_n\bigl\|\pi_n(gh)-\pi_n(g)\pi_n(h)\bigr\|_{2}=0.$$
If in addition
\(
\lim_n \bigl\|\pi_n(g)-\pi_n(h)\bigr\|_{2}=\sqrt2
\)
whenever $g\neq h$, then the sequence is a \emph{hyperlinear approximation} of~$\Gamma$. A group admitting such an approximation is said to be \emph{hyperlinear}, or \emph{Connes embeddable}.
\end{definition}

Sofic groups arise when $U(n)$ is replaced by finite symmetric groups $\Sym (n)$ equipped with the normalized hamming metrics, and are in particular hyperlinear \cite{elek2005hyperlinearity}.
The subject posses rich history, we refer to \cite{pestov2008hyperlinear, ThomICM} for excellent surveys.  

Our main result provides new, particularly simple candidates for non-hyperlinear groups.
To formulate it, we need the notion of  Hilbert--Schmidt stability.
Define the generalized Hilbert--Schmidt metric $d_2$ on the disjoint union $\bigsqcup_{m \in \N} U(m)$  as follows.
For $n\leq m$, $x \in U(n), y \in U(m)$, let
\begin{equation*}
     d_2(x,y) = \| x \oplus 0 _{m-n} - y \|_{2,m}
     \end{equation*}
Where $0_{m-n}$ is the $(m-n) \times (m-n)$-square null matrix\footnote{$d_2$ is indeed a metric, as can be verified similarly to \cite[Lemma A.1]{BC}}.
\begin{definition}
     A group $\Gamma$ is \emph{Hilbert--Schmidt stable} \cite{BL}  if for every asymptotic representation $\pi_n:\Gamma\to U(d_n)$, there exists a sequence $D_n \geq d_n$ and a sequence of honest representations $\rho_n:\Gamma\to U(D_n)$ such that $d_2(\pi_n,\rho_n)\to 0$.
If $D_n$ can be chosen to be $d_n$, we say $\Gamma$ is \emph{strictly Hilbert--Schmidt stable}\footnote{In the literature, the definition we gave for stability is usually referred to as \emph{flexible stability}
and stability refers to strict stability instead. We chose this terminology as with time, it has become clear that flexible stability is better suited for most applications.}. 
\end{definition}


   Recall that $\Gamma$ has \emph{property (T;FD)} of Lubotzky-Zimmer \cite{LZ} if finite dimensional $\Gamma$-representations have uniform spectral gap. 

\begin{theorem}\label{thm:main-non-hyperlinear}
Let $\Gamma$ be a group with  property (T;FD). Assume $\Gamma$ has a central extension of the form 
\[
1\longrightarrow \Z\longrightarrow\widetilde\Gamma\longrightarrow\Gamma\longrightarrow1
\]
where $\widetilde{\Gamma}$ is a group with finite abelianization. If $\Gamma$  is Hilbert--Schmidt stable, then there exists $N \in \N$ such that the finite central extension:
\[
1\longrightarrow \Z/N  \longrightarrow\widetilde\Gamma / (N\cdot \Z)\longrightarrow\Gamma \longrightarrow1
\]
is not hyperlinear.
\end{theorem}

   We shall see that Theorem \ref{thm:main-non-hyperlinear} applies to higher rank lattices in semisimple groups (Theorem \ref{thm:main_S_arith}), even in absence of (T).
This in turn yields surprisingly simple examples-- \emph{amalgamated free products of virtually free groups over finite index subgroups}, which have (T:FD) but not (T), as we now discuss in detail.
   
\subsection*{A candidate for non-hyperlinearity} \label{sec:prologue}


Let $\Lambda$ be a finitely generated group with generating set $S \subset \Lambda$. Let $X_{\mathrm{FD}}(\Lambda)$ denote the \emph{representation variety}\footnote{More precisely, it is  a countable union of real algebraic varieties.} of all of finite
dimensional unitary representations $\Lambda\to U(n)$, $n\in \N$.
Given $\pi, \tilde \pi \in X_{\mathrm{FD}}(\Lambda)$ their distance is defined by
\[
       d_2(\pi,\tilde{\pi})= \max_{s\in S} d_2(\pi(s), \tilde\pi(s)),
\]

Consider the modular group $ \mathrm{SL}_{2}(\Z)$, and
let ${B}_{+}$, ${B}_{-}$ be the \emph{upper and lower Iwahori subgroups} \cite[Section 1.2]{serre2002trees}:
$${B}_+ = \big\{  \left(\begin{array}{cc} 
a & b\\
2c & d
\end{array}\right) \big \}, \; B_-=\big\{  \left(\begin{array}{cc} 
a &2 b\\
c & d
\end{array}\right) \big \} \leq \SL_2(\Z). $$
Note that these are the pullbacks of the upper (resp. lower) triangular subgroups of $\SL_2(\mathbb F_2)$ under the congruence quotient map from $\SL_2(\Z)$.
Thus, they are of index $3$.
The   bijection
\[
\left(\begin{array}{cc}
a & b\\
2c & d
\end{array}\right)
\overset{\sigma}{\longleftrightarrow}
\left(\begin{array}{cc}
a & 2b\\
c & d
\end{array}\right)
\]
 is a group isomorphism  $\sigma:{B}_{+} \cong {B}_{-}$. It is implemented as conjugation by the diagonal matrix $t=\mathrm{diag}(\sqrt 2,1/\sqrt 2)$, which is in the commensurator, but not the normalizer of $\SL_2(\Z) \leq \SL_2(\R)$.
 The amalgamated free product $\mathrm{SL}_2(\Z )*_{ B}\mathrm{SL}_2(\Z )$ is then isomorphic to $\mathrm{SL}_2(\Z[1/2])$ \cite[\S II.1, Cor. 2]{serre2002trees}.

Consider the free product $\Lambda=\mathrm{SL}_2(\Z)*\mathrm{SL}_2(\Z)$, and the corresponding representation variety $X_{\mathrm{FD}}(\Lambda)$ with respect to some fixed finite generating set. 
For $\pi \in X_{\mathrm{FD}}(\Lambda)$ we can consider its restrictions $\pi_+,\pi_-$ to the first and second factors respectively.
Consider the following diagram:
\[\begin{tikzcd}
	{{B}_+} & {\mathrm{SL}_2 (\mathbb Z)} \\
	{} & {} & {U(d)} \\
	{{B}_-} & {\mathrm{SL}_2 (\mathbb Z)}
	\arrow[hook, from=1-1, to=1-2]
	\arrow["\sigma"{description},leftrightarrow, from=1-1, to=3-1]
	\arrow["{\pi_+}", from=1-2, to=2-3]
	\arrow[hook', from=3-1, to=3-2]
	\arrow["{\pi_-}"', from=3-2, to=2-3]
\end{tikzcd}\]
If this diagram commutes, we say that $\pi$
is \textit{compatible}. The space of  compatible representations is a subvariety of $X_{\mathrm{FD}}(\Lambda)$ identified with $X_{\mathrm{FD}}(\mathrm{SL}_2(\Z[1/2]))$.

Fix a finite generating set $S_{+}$ for ${B}_{+}$, which can be taken to be:
\[
\left\{ \left(\begin{array}{cc}
1 & 0\\
2 & 1
\end{array}\right),\left(\begin{array}{cc}
1 & 1\\
0 & 1
\end{array}\right),\left(\begin{array}{cc}
-1 & 0\\
0 & -1
\end{array}\right)\right\} \subset {B}_+ .
\]
The \emph{local defect} of $\pi$ is then measured by:
\[
\mathrm{def}(\pi)=\max_{s \in S_+} d_2(\pi_{+}(s),\pi_{-}(\sigma(s)))
\]
The \emph{global defect} of $\pi$ is measured by
\[
    \mathrm{D}(\pi)=\inf\{d_2(\pi,\tilde{\pi})\mid \tilde{\pi}\in X_{\mathrm{FD}}(\Lambda) \text{ is compatible}\}.
\]
It is not hard to see that $\mathrm{def}(\pi)\leq 2\mathrm{D}(\pi)$.  We pose the following problem:
\begin{question}
    \label{problem!}
Does there exist a  function $f:\R_+\to \R_+$ with $\lim_{x\to 0} f(x)=0$ such that for any  $\pi\in X_{\mathrm{FD}}(\Lambda)$ one has
\[
\mathrm{D}(\pi)\leq f(\mathrm{def}(\pi))
\]
that is, are representations of $\SL_2(\Z)$ which almost factor through a congruence quotient $SL_2(\mathbb Z/n)$, for some odd $n$, close to congruence representations?
\end{question}

The later formulation follows from the congruence subgroup property of $\Gamma = \SL_2(\Z[1/2])$: finite dimensional unitary representations of $\Gamma$ correspond to finite dimensional unitary representations of $\SL_2(\Z)$ which factor through an uneven congruence quotient.

It follows from Theorem \ref{thm:main-non-hyperlinear} that a positive answer to Question \ref{problem!} implies the existence of a non-hyperlinear group.
In fact:


\begin{corollary}\label{cor:SL_2}
    Assume a positive answer to Question \ref{problem!}, then there is a finite central extension
    \[
1\longrightarrow \Z/N\longrightarrow\widetilde\Gamma\longrightarrow \SL_2(\Z[1/2])\longrightarrow1
\]
    which is not hyperlinear.
\end{corollary}

Notably, such a group $\tilde \Gamma$ is an amalgamated product of virtually free groups over a finite index subgroup.
By a theorem of Deligne \cite{Deligne} and Raghunathan \cite{Rag_torsion} regarding central extensions of $S$-arithmetic groups, $\widetilde{\Gamma}$ is not residually finite.
Combined with Margulis' normal subgroup theorem \cite{margulis1991discrete}, it is not residually amenable, making it an intriguing candidate for a non-hyperlinear group.
The same is true for $\SL_2(\Z[1/p])$, $p$ prime, and a wide class of $S$-arithmetic groups (see Theorem \ref{thm:main_S_arith}).


\

The high level strategy of the proof of Theorem \ref{thm:main-non-hyperlinear} is as follows: Assuming $\widetilde{\Gamma}$ is hyperlinear, we obtain \emph{asymptotically projective representations} (Definition \ref{def:asym_proj}) arising from the central extension -- thought of as ``deformations" of the regular representation of $\Gamma$.
The main technical core is contained in Sections \ref{sec:local_rigidity} and \ref{sec:non-hyperlinear}, which establish new local rigidity phenomenona for asymptotic representations, as well as asymptotically projective representations, under the assumption of Hilbert--Schmidt stability.
The deformation posited by hyperlinearity in conjunction with our new rigidity results then lead to a contradiction.

The general idea of using deformations arising from hyperlinear central extension in conjunction with property (T) to disprove stability was already used in \cite{ISW,Dog}.
However, proving Theorem \ref{thm:main-non-hyperlinear}  requires genuinely new techniques.
This is because in absence of property (T), one cannot apply the rigidity phenomena used in \cite{ISW} which involve infinite dimensional representations.
We overcome this difficulty by by deducing the relevant rigidity phenomena for asymptotically projective representations by upgrading property (T:FD) using Hilbert--Schmidt stability (Proposition \ref{prop:rigidity_almost_rep} and Theorem \ref{thm:NPS_analog}).
This introduces extra intricacies, in particular, a Connes type lemma  (Lemma \ref{lem:Connes_trick}) is used.

\subsection*{Hyperfinite stability, asymptotic spectral gap and character rigidity}

Beyond Theorem \ref{thm:main-non-hyperlinear}, we investigate asymptotic spectral gap properties of general higher rank lattices also in relation to character rigidity.
 Recall that a \emph{character} on a group $\Gamma$ is a positive-definite, conjugation-invariant function $\varphi:\Gamma \to \C$ with $\varphi(e)=1$ which cannot be decomposed as a convex combination of such functions (see \S\ref{subsec:chars}). 
 A group $\Gamma$ is said to be \emph{character rigid} if every character on $\Gamma$ is either induced by a finite dimensional representation, or vanishes off the center  $Z(\Gamma)$ \cite[Definition 1.1]{char-rig-nonuniform}.
 
 A well-known conjecture due to Connes (see  \cite{Jo00} and \cite{houdayer2021noncommutative}) states that every irreducible lattice $\Gamma$ in a higher rank semisimple group $G$ is character rigid.
This conjecture has been established in many cases, most notably when $G$  has a property~(T) factor \cite{bekka2006operator,peterson2015character, boutonnet2021stationary,BBHP,BBH,creutz2024character},
or when $\Gamma$ is non-uniform \cite{char-rig-nonuniform, PT}. 
Despite such advances, the conjecture remains open.


Our second result gives several reformulations of character rigidity in the setting of higher rank lattices:
It is equivalent to a spectral gap property of asymptotic representations, termed \emph{robust property (T:FD)} (Definition \ref{def:robust}), which in turn is equivalent to an amenable  form of stability, termed \emph{hyperfinite Hilbert--Schmidt stability} (Definition \ref{def:hyp_HS_stab}).
Notably, it  follows that \emph{character rigidity is witnessed on a finitary level}.
We refer to Sections \ref{sec:hyp_HS_stab} and \ref{sec:rob} for precise definitions and for generalities regarding hyperfinite Hilbert--Schmidt stability and property \textnormal{(T;FD)}$_{\mathrm{rob}}$.

\begin{theorem}\label{thm:main_char_rig_INTRO}
Let  $G$ be a center-free connected semisimple Lie group without compact factors and with real rank $\ge 2$.  Let $\Gamma\leq G$ be an irreducible lattice  that has property~\textnormal{(T;FD)}. The following conditions are equivalent:
\begin{enumerate}
   \item\label{it:hfstab} $\Gamma$ is hyperfinitely Hilbert--Schmidt stable;
   \item\label{it:char} $\Gamma$ is character rigid;
   \item\label{it:tfdrob} $\Gamma$ has property~\textnormal{(T;FD)}$_{\mathrm{rob}}$;
   \item\label{it:fdlimits} Characters of $\Gamma$ are pointwise limits of normalized traces of finite dimensional representations.
\end{enumerate}
\end{theorem}

In particular, if $\Gamma$ is Hilbert--Schmidt stable, then it is character rigid.

The proof of Theorem \ref{thm:main_char_rig_INTRO} crucially relies \emph{charmenability}, a strong property established in the work of \cite{BBHP}, which shows that characters of $\Gamma$ that do not vanish outside the center must be von Neumann amenable.
With that said, showing the equivalences requires several results of independent interest. This includes a non-commutative version of Schramm's theorem  \cite{Schramm}  on hyperfinite graph limits  (see Theorem \ref{thm:vN_Schramm}).
The equivalence $(1) \Leftrightarrow (4)$  follows from a general character-theoretic criterion for hyperfinite Hilbert--Schmidt stability (Theorem \ref{thm:HS_for_non_amenable_grps}), building on \cite{HS_grp}. 
For the remaining equivalences, we make use of charmenability:
Somewhat paradoxically, the core idea in our approach to $(1) \Leftrightarrow (2) \Leftrightarrow (3)$ is to leverage amenability in order to derive spectral gap, relying on Connes' theorem \cite{Connes} (see \S\ref{sec:char_rigidity}).
Along the way, we provide a general framework for showing that spectral gap passes to limits, generalizing \cite{levit2023spectral} (see \S\ref{sec:rob}).

In contrast to the situation of property (T:FD), we record a short proof due to U. Bader (Theorem \ref{thm:op_main}) that property (T) and its robust variant are in fact equivalent, generalizing \cite{Manuilov2007OnAR}.

\subsection*{Organization of the paper.}
Section~\ref{sec:preliminaries} collects preliminaries on von Neumann algebras and characters. 
Section~\ref{sec:local_rigidity} develops local rigidity for asymptotic representations via Hilbert--Schmidt stability.
Section~\ref{sec:non-hyperlinear} proves the main Theorem~\ref{thm:main-non-hyperlinear} and deduces results for $S$-arithmetic groups.  
In Section~\ref{sec:hyperfinite} we introduce hyperfinite stability and prove the non-commutative analogue of Schramm's theorem (Theorem~\ref{thm:vN_Schramm}).  In \ref{sec:hyp_HS_stab}, the notion hyperfinite-stability is introduced, along with its character-theoretic counterpart.
In \ref{sec:rob}, we study asymptotic spectral gap properties in abscense of stability.
The equivalences in Theorem~\ref{thm:main_char_rig_INTRO} are then proved in Section~\ref{sec:char_rigidity}.

\subsection*{Acknowledgments}
We deeply thank Alex Lubotzky, Uri Bader, for their continuous encouragement and advice regarding various parts of this work.
We are grateful Thomas Vidick for his immense support, and discussions on hyperfiniteness which led to Definition \ref{def:hyperfinite tuple}.
We sincerely thank David Gao and Adrian Ioana for providing parts of the proof of Theorem \ref{thm:vN_Schramm} and Proposition \ref{prop:hyp_stab_prod}.
We thank Francesco Fournier-Facio for helpful comments on the paper.
Finally, we thank Yuval Gorfine and Michael Glasner for their friendship and many advice regarding spectral gaps.

\subsection*{Funding}
AD was supported by a Clore Scholars grant and ERC grant no. 882751. IV was supported by  NSF postdoctoral fellowship grant DMS-2402368.

\subsection*{Declarations}
No AI tools were used by the authors in the research or writing processes of this article.

\section{Preliminaries: Traces on groups and von Neumann algebras}
\label{sec:preliminaries}

\subsection{Tracial von Neumann algebras}\label{subsec:von-Neumann-alg}

We refer to the book \cite{Popa} for information on the theory of tracial von Neumann algebras. Here, we recall the bare necessities.
 A \emph{tracial von Neumann algebra} is a pair $(M,\tau)$ where $M$ is a von Neumann algebra, and $\tau:M\to \C$ is a normal and faithful tracial state (simply referred to as a \emph{trace}). Normality in this context refers to continuity in the $\sigma$-weak topology.
 $M$ is said to be \emph{finite} if it admits a normal and faithful trace.
 $M$ is said to be a \emph{factor} if the center of $M$ is one-dimensional, and is a II$_1$\emph{-factor} if $M$ is an infinite dimensional finite factor.  The trace on a finite factor is unique. The group of unitary elements in $M$ is denoted by $U(M)$.

The map $\langle x, y \rangle_\tau\mapsto \tau(y^*x)$ defines an inner product on $M$, which in turn induces a norm $\| x \|_{2,\tau} = \tau(x^*x)^{1/2}$ and as a result a metric $d_{2,\tau}(x,y)$.
When clear form the context, we will simply denote this norm by $\| \cdot \|_2$, and refer to it as the \emph{$2$-norm}.
The completion of $M$ with respect to this norm is a Hilbert space  denoted by $L^2(M) = L^{2}(M,\tau)$.
We have a canonical embedding $M\hookrightarrow L^2(M,\tau)$, and the anti-linear map $x \mapsto x^*$ extends to an anti-linear isometry $J: L^2(M, \tau) \to L^2(M,\tau)$, called the \emph{canonical conjugation}. 
Further, Left (and also right) multiplication in $M$ extends to bounded operators on $L^2(M)$, making it the \emph{standard $M$-bimodule}, and we have the formula $x \cdot \xi \cdot y^* = x\cdot (JyJ) \cdot\xi$ for all $x,y\in M, \xi \in L^2(M)$.

If $N\leq M$ is a von Neumann subalgebra (with $1_M\in N$), then $L^2(N,\tau\mid_N)$ is naturally embedded as a closed subspace in $L^2(M,\tau)$. As a result, there is an orthogonal projection map, which when restricted to $M$, is denoted by $\mathbb{E}_N:M\to N$. This map is called the \emph{conditional expectation} onto $N$.

    Even though the language of general tracial von Neumann algebras will be invaluable to our study, one of the main examples we will be working with is simply the finite dimensional matrix algebras $M_n(\C)$ with the normalized matrix trace.
    In this case, the $2$-norm is simply the normalized Hilbert--Schmidt norm mentioned in the introduction.

There are several ways to study approximations of von Neumann algebras by finite dimensional ones, see \cite{BO} for this beautiful subject.
For a  set $F\subset M$ and $\epsilon>0$, an $(F,\epsilon)$-\emph{hyperfinite approximation} of $(M,\tau)$ is a finite dimensional (unital) $*$-subalgebra $Q \leq M$ such that the conditional expectation $\mathbb{E}_Q:M\to Q$ satisfies $\| \mathbb E_Q(x) - x \|_2 \leq \epsilon$ for all $x \in F$. 
$(M,\tau)$ is said to be \emph{hyperfinite} if it admits an $(F,\epsilon)$-hyperfinite approximation for any finite set $F\subset M$ and any $\epsilon>0$. 
Every finite dimensional von Neumann algebra is hyperfinite in an obvious sense. 
Remarkably, Murray and von Neumann proved there  exists a unique hyperfinite II$_1$ factor, denoted $\mathcal R$. Connes proved the fundamental fact that \emph{amenability} and hyperfiniteness are equivalent \cite{Connes}.

A more permitting approach for finite dimensional approximation is given by the following notion. $(M,\tau)$ is said to be \emph{Connes embeddable} if for any finite $F\subset M$ and $\epsilon>0$ there exists a finite dimensional tracial von Neumann algebra $Q$, and a linear, almost trace preserving $*$-map $\pi: M\to Q$ which is $(F,\epsilon)$-multiplicative, see \cite[Section 4]{Dog_Thesis} for a discussion regarding this notion. 
A group $\Gamma$ is hyperlinear if and only if its group von Neumann algebra $L\Gamma$ is Connes embeddable.  Only recently, non-Connes embeddable finite factors have been established \cite{MIP*=RE}. These finite factors, however, are not group von Neumann algebras.

A key tool in the study of von Neumann algebras in general, and of Connes embeddability in particular, is of \emph{tracial ultraproducts}. 
Throughout this paper, we will fix a non-principal ultrafilter $\omega$ on $\N$.
We will denote the ultralimit of a bounded sequence of scalars $x_n$ with respect to $\omega$ by $\ultlim{n} x_n$.
Let $(M_n, \tau_n)$ be a sequence of tracial von Neumann algebras (which for most of our purposes will be $M_n = M_{d_n}(\C)$, for some unbounded sequence $d_n \in \N$). 
The \emph{tracial ultraproduct} is the quotient $\ultprod{n} M_n = \left(\prod_{n \in \N} M_n \right) / \mathcal I_\omega$, where $\prod_{n \in \N} M_n = \{ (x_n)_n | ; \sup_n \| x_n \|_{\op} < \infty \}$ and  $\cI_\omega = \{ (x_n) | \ultlim{n} \| x_n \|_{2,\tau_n} = 0\}$.
Together with the natural ultralimit trace, denoted by $\tau_\omega$, the algebra $\prod_{n \to \omega} M_n$ is a tracial von Neumann algebra.
A  tracial von Neumann algebra $(M,\tau)$ is Connes embeddable if and only if it embeds in a tracial ultraproduct of the matrix algebras $M_n(\C), n \in \N$; equivalently, if it embeds into $\mathcal R^\omega=\prod_{n\to \omega}\mathcal{R}$. 

\subsection{Characters  and traces of groups}\label{subsec:chars}
Let $\Gamma$ be a (discrete) group. A tracial representation is a triple $(M,\tau,\pi)$ where $(M,\tau)$ is a tracial von Neumann algebra, and $\pi:\Gamma \to U(M)$ is homomorphism satisfying $\pi(\Gamma)''=M$. Two tracial representations are called \emph{quasi-equivalent}, if there is trace preserving isomorphism between the von Neumann algebras which intertwines the respective representations. It is easily verified that the $\varphi=\tau\circ \pi$ is a \emph{trace} on $\Gamma$;
by definition, the means that $\varphi$ is
\begin{enumerate}
    \item  positive-definite: $\sum_{i,j}\alpha_i \bar{\alpha}_j \varphi(x_j ^{-1}x_i)\geq 0$ for any $n\in \N$, $x_1 ,...,x_n\in \Gamma$, and $\alpha_1 ,...,\alpha_n \in \C$,
    \item normalized: $\varphi(e)=1$,
    \item conjugation-invariant: $\varphi(xyx^{-1})=\varphi(y)$ for all $x,y\in \Gamma$. 
\end{enumerate}
The set of traces on $\Gamma$, denoted by $\Tr{\Gamma}$, is a compact convex space with respect to the topology of pointwise convergence.
In fact, it is a Choquet simplex, and its extreme points are called \emph{characters} of $\Gamma$.
Note that, equivalent tracial representations give rise to the same trace. 
The GNS construction gives a converse, namely, every $\varphi \in \Tr{\Gamma}$ is of the form above, for a unique (up to quasi-equivalence) tracial representation $\pi_\varphi: \Gamma \to U(M_\varphi)$. 
In fact, a theorem of Thoma shows that $M_\varphi$ is a factor if and only if $\varphi$ is a character \cite{thoma1964unitare}.
We refer to \cite[Chapter 11]{BdlH} for all this and more information on characters and traces on groups.

A trace $\varphi$ is called \emph{finite dimensional} if it is of the form $\frac{1}{\dim\pi}\mathrm{tr}\circ{\pi}$ for some finite dimensional unitary representation $\pi$ of $\Gamma$. 
A trace $\varphi$ is called \emph{von Neumann amenable} \cite{BBHP}, or \emph{uniformly amenable} \cite{Brown}, if the corresponding tracial von Neumann algebra $(M,\tau)$ is amenable.
Thus, a von Neumann amenable character is either finite dimensional, or else it is of the form $\tau\circ\pi$ for some homomorphism $\pi:\Gamma\to U(\mathcal{R})$

\subsection{Convergence of marked von Neumann algebras}\label{subsec:marked}
Let $(M, \tau)$ be a tracial von Neumann algebra, and $x^1, \dots, x^k \in  U(M)$ be a tuple of unitaries that generate it.
We refer to the data $( M ; x^1, \dots, x^k)$, also denoted by $(M; \bar x)$ as a \emph{marked von Neumann algebra}, or simply refer to $\bar{x}=(x^1, \dots, x^k)$ as a \emph{tuple}, when $( M, \tau)$ is clear from the context.
Two marked von Neumann algebras $(M; x^1, \dots x^k)$, $(N; y^1,\dots y^k)$ are isomorphic if there is a trace preserving $*$-isomorphism $f: M \to N$ such that $f(x^i) = y^i$ for $1 \leq i \leq k$.

The set of isomorphism classes of $k$-marked von Neumann algebra admits a natural topology. We say that a sequence  $( M _n; x_n^1, \dots, x_n^k)$ \emph{converges} to  $( M ; x^1, \dots, x^k)$ if
\[
    \tau_n(p(x_n^1 ,...,x_n^k))\xrightarrow[n \to \infty]{} \tau(p(x^1 ,...,x^k))
\]
for any $*$-polynomial  $p$ in $k$ non-commuting variables. 

Equivalently, one can view the tuples as specified by a tracial representation of the free group $\langle s_1, \dots, s_k \rangle = \mathbb F_k$ on $k$-generators:
Each tuple $( M; x^1, \dots, x^k)$ gives a representation $\pi: \mathbb F_k \to U( M)$, by taking $\pi(s_i) = x^i$ for $i \leq k$.
Conversely, one can associate a tuple to every representation $\pi: \mathbb F_S \to U( M)$ by taking $x^i = \pi(s_i)$.
Under this identification, there is a map taking a $k$-marked von Neumann algebra $( M; \bar x)$ to a trace $\varphi \in \Tr{\mathbb F_k}$, defined as $\varphi = \tau \circ \pi$, for the associated representation $\pi: \mathbb F_k \to U( M)$ and the trace $\tau$ on $M$.
It is straightforward to see using the GNS-construction that this map induces a bijection between the isomorphism classes of $k$-marked von Neumann algebras, and the simplex of traces on the free group $\Tr{\mathbb F_k}$.
The latter is a compact space with respect to pointwise convergence of functions, and it is straightforward to see that the two topologies coincide.  
Note that this notion of convergence is also referred to as  \emph{convergence in moments}, or \emph{weak convergence}, as is common in free probability theory, where it is heavily investigated (see for example \cite[Definition 5.2.5]{Anderson_Guionnet_Zeitouni_2009}). 
We remark  that $\Tr{\mathbb F_k}$ is a metrizable Choquet simplex, and in fact a Poulsen simplex \cite{OSV}. This means that the set of tuples which generate factors is dense in the space of marked von Neumann algebras.

Let  $(\Gamma;\gamma_1 ,...,\gamma_k)$ be a  \emph{marked group}, that is, a group endowed with a generating $k$-tuple. 
Every representation $\pi \to U (M)$, where $M$ is a tracial von Neumann algebra and $\pi(\Gamma)''=M$, gives rise to a marked von Neumann algebra $(M;\pi(\gamma_1),...,\pi(\gamma_n))$. 
The set of all such marked von Neumann algebras is a closed face of the space of all marked von Neumann algebras; indeed, it corresponds to the closed face of $\Tr{\F_k}$ consisting of all traces factoring through the quotient $\F_k\to \Gamma$.


For every normal element $x\in M$, we can consider its \emph{spectral measure}
$\mu=\mu_{x}\in \mathrm {Prob} (\C) $ defined by:
\[
\mu(f)=\tau\left(f(x)\right),\,\,\,\,\,\,f\in C_{c}(\C)
\]
where $f(x)$ is defined by continuous functional calculus arising from the embedding $M\subset \mathcal {B} (L^2(M,\tau))$.
The following lemma relates convergence in moments to weak-$*$ convergence of spectral measures.  

\begin{lemma}\label{lem:convergence of spectral measures}
Let $(M_n;\bar x_n)$ be a sequence of $k$-marked von Neumann algebras converging to $(M;\bar x)$, and let $p$ be a $*$-polynomial in $k$ variables. Assume that for all $n$, $p(\bar{x}_n) \in M_n$ and $p(\bar x)\in M$ are self-adjoint, and consider the spectral measures $\mu_n:=\mu_{p(\bar{x}_n)}$ and $\mu:=\mu_{p(\bar{x})}$. Then $\mu_n\to \mu$ in the weak-$*$ topology of probability measures on $\R$. 
\end{lemma}

\begin{proof}
Since the the tuples $\bar x_n, \bar x$ all consists of unitaries, we have that for all $n$, $\| p(\bar x_n)\| \leq C$ and $\| p(\bar x) \| \leq C$ , where $C = \| p \|_1$ is the sum of the absolute values of the coefficients of $p$.
As such, the measures $\mu_n, \mu$ are all supported on the closed interval $[-C,C]$.
By the Stone-Weierstrass theorem, to prove that $\mu_n$ weakly converge to $\mu$, it is enough to show that for every univariate polynomial $f$, we have $\int f d\mu_n \to \int f d\mu $.
However, by definition of spectral measures, we have:
$$ \int f d\mu_n = \tau_n( f(p(\bar x_n))) = \tau_n(f\circ p(\bar x_n)) \to \tau(f\circ p(\bar x)) = \int f d\mu $$
where we used  convergence on the moment corresponding to the polynomial $f\circ p$. 
\end{proof}

\section{Local rigidity of asymptotic representations}\label{sec:local_rigidity}

This section is dedicated to proving a local-rigidity property for asymptotic representations of groups $\Gamma$ with property (T;FD), under the strong assumption of Hilbert--Schmidt stability:
We show that there exists some $\epsilon>0$ such that if two asymptotic representations of $\Gamma$ are asymptotically $\epsilon$-close to each other on the generators of $\Gamma$, then they are in fact asymptotically conjugate (the final statement will allow to compare asymptotic representations of different dimensions).
We first recall the local rigidity phenomenon that holds for finite dimensional unitary representations under the assumption of (T;FD) (see \cite{dlHRV}, \cite{Rapinchuk}).


\begin{lemma}[local rigidity of finite dimensional representations]
\label{lem:T_FD_rigidity_rep}
If  $\Gamma = \langle S \rangle$ has property (T;FD),  then for every $0< \delta<1/2$ there exists $\epsilon> 0$ such that the following holds:
For every $d$, and $M= M_{d}(\C)$, and a nonzero projection $p \in M$, if  $\pi:\Gamma\rightarrow U(p M p)$, $\rho: \Gamma \to U( M )$ are  unitary representations such that
$$\| \pi(s) - p\rho(s) p\|_{2,M} / \|p \|_{2,M}< \epsilon \text{ for all $s \in S$,}$$
then there exists $\xi \in p M$ with $ \| \xi -p \|_{2,M}  / \| p \|_{2,M}\leq \delta $, $\| \xi\|_{2,M} / \| p\|_{2,M} = 1$ and such that 
$$ \pi(g)\xi\rho(g)^* = \xi   \text{ for all $g \in \Gamma$.}$$
Furthermore, we have $\sup_{g \in \Gamma}\| \pi(g)  p\rho(g)^* - p\|_{2,M} / \|p \|_{2,M} \leq 2\delta$.
\end{lemma}

\begin{proof}
    Since $\Gamma$ has property (T;FD), the following holds: For every $\delta>0$ there exists $\epsilon> 0$ such that if $\alpha:\Gamma \to U(\cH)$ is a unitary representation with $\dim \cH < \infty$, and there is a unit vector $\xi \in \cH$ which is $(S,\epsilon)$-invariant, then there exists a unit vector $\eta \in \cH$ which is $\Gamma$-invariant and such that $\| \xi - \eta \| < \delta$ (See \cite[Proposition 1.1.9]{BdlH}, or \cite[Corollary 15.1.4]{Popa}, the proof is the same as in the case of property (T)).

Given $1/2>\delta>0$, let $\epsilon > 0$ be as in the conclusion of (T;FD) in the paragraph above.
Let $\pi: \Gamma \to U(pMp),\rho: \Gamma \to U(M)$ be given such that $\| \pi(s) - p\rho(s) p\|_{2,M} / \|p \|_{2,M}< \epsilon/2$ for $s \in S$.
Consider the Hilbert space $\cH = pL^2(M)$, with the norm given by $\|  x\| = \| x \|_{2,M} / \|p \|_{2,M}$ for $x \in \cH$.
We can then define the "left-right" unitary representation $\alpha: \Gamma \to U(\cH)$ by $\alpha(g) \xi =\pi(g) \xi \rho(g)$.
We claim that $\eta = p \in \cH$ is a $(S,\epsilon)$-invariant unit vector of $\alpha$. 
Indeed, we have for all $s \in S$:
\begin{align*}
    \| \alpha(s)p - p\|_{2,M}^2 &=  \| \pi(s)p \rho(s)^*- p\|_{2,M}^2 \\
    &= \| \pi(s) p \rho(s)^* \|_{2,M}^2 + \| p \|_{2,M}^2 - 2\Ra{\tau_{M}(p\pi(s)p\rho(s)^*)} \\
    &= 2 \left(\tau_{M}(p) - \Ra{\tau_{M}(p\pi(s)p\rho(s)^*p)} \right)\\
    &= 2 \left(\Ra{\tau_{M}(p\pi(s)(\pi(s)^* - p\rho(s)^*p))} \right)\\
    &\leq_{C.S} 2 \| p\pi(s) \|_{2,\tau_{M}} \| \pi(s)^* - p\rho(s)^*p \|_{2,M}\\
    &\leq \epsilon \cdot  \| p \|_{2,M} \|p \|_{2,M}\\
    &
= \epsilon \cdot  \| p \|_{2,M}^2.
\end{align*}

Consequently, there exists a unit vector $\xi \in \cH $ such that  $\pi(g)\xi \rho(g) = \xi$ for all $g \in \Gamma$ and $\| \xi - p \|_{2,M} \leq \delta \| p \|_{2,M}$. The furthermore conclusion follows.
\end{proof}

We need a strengthening of the above lemma, which will guarantee the existence of an intertwiner $\xi$ as above that has \emph{bounded operator norm}.

\begin{lemma}[local rigidity of finite dimensional representations, strengthened]
\label{lem:T_FD_rigidity_rep_strong}
Under the same assumptions and notation of Lemma \ref{lem:T_FD_rigidity_rep}, there exists $\xi \in p M$ with $ \| \xi -p \|_{2,M}  / \| p \|_{2,M}\leq \delta $, $\| \xi\|_{2,M} / \| p\|_{2,M} = 1$ \textbf{and} $\| \xi \|_{op} \leq 2$, such that:
$$ \pi(g)\xi\rho(g)^* = \xi   \text{ for all $g \in \Gamma$.}$$
\end{lemma}

\begin{proof}
    Let $1/2 > \delta> 0$, by Lemma \ref{lem:T_FD_rigidity_rep} there exists $\epsilon$ such that if 
$$\| \pi(s) - p\rho(s) p\|_{2,M} / \|p \|_{2,M}< \epsilon \text{ for all $s \in S$,}$$
    then we have $\sup_{g \in \Gamma}\| \pi(g)  p\rho(g)^* - p\|_{2,M} / \|p \|_{2,M} \leq \delta$. Let $K$ be the convex hull of the set $\{ \pi(g) p \rho(g)^* | \; g \in \Gamma \}$ in $pM$. Note that $K$ is contained in the ball of radius $\delta$ around $p$ in the appropriately normalized $2$-norm.
    Further, we have $\| \pi(g) p \rho(g)^*\|_{\op} \leq 1$ for all $g \in \Gamma$, so that $K$ is bounded in operator norm by $1$.
    Consequently, letting $\eta$ be the unique vector of minimal $2$-norm in $K$ , we get an intertwiner $\eta$ such that $\| \eta \|_{\op} \leq 1$, and $ \| \eta -p \|_{2,M}  / \| p \|_{2,M}\leq \delta$.
    By the last inequality we have $|\| \eta \|_{2, M}  / \| p \|_{2,M} - 1| \leq \delta$, so that $\| \eta  \|_{2,M} / \| p \|_{2,M} \geq 1-\delta \geq  1/2$.
    Setting $\xi = (\| p\|_{2,M} / \| \eta \|_{2,M}) \cdot\eta $, we have $\| \xi \|_{op} \leq 2$ is an intertwining unit vector, and:
    \begin{align*}
        \| \xi - p \|_{2,M} / \| p \|_{2,M} &\leq \| \xi - \eta \|_{2,M} / \| p \|_{2,M} + \| \eta - p \|_{2,M} / \| p \|_{2,M} \\
        &\leq (\|p\|_{2,M} / \|\eta \|_{2,M} - 1)\| \eta  \| / \| p \|_{2,M} + \delta \\
        &\leq \delta (1 + \delta) + \delta.
    \end{align*}
    Thus, we get the lemma with $\delta^2 + 2 \delta$, which finishes the proof.
    
\end{proof}

We arrive at the main proposition, which adapts the above lemma to the case of asymptotic representations.

\begin{proposition}[local rigidity for approximate representations]
\label{prop:rigidity_almost_rep}
Assume $\Gamma = \langle S \rangle$ has property (T;FD) and is Hilbert--Schmidt stable.

Then for every $1/2 > \delta>0$ there exists $\epsilon> 0$ such that the following holds:
For every sequence $d_n$, and $M_n = M_{d_n}(\C)$, and projections $q_n \in M_n$ with $\liminf_n \| q_n\|_{2,M_n} > 0$, if  $\pi_n:\Gamma\rightarrow U(q_n M_n q_n)$, $\rho_n: \Gamma \to U( M_n )$ are asymptotic representations (with the corresponding normalized Hilbert Schmidt norms) such that for almost every $n$:
$$\| \pi_n(s) - q_n\rho_n(s) q_n\|_{2,M_n} / \|q_n \|_{2,M_n}< \epsilon \text{ for all $s \in S$,}$$
then there exist vectors $\xi_n \in q_nM_n$ of uniformly bounded operator norm such that $ \| \xi_n -q_n \|_{2,M_n}  / \| q_n \|_{2,M_n}\leq \delta$, $\| \xi_n \|_{2,M_n} / \| q_n \|_{2,M_n} =1$ and for  every $g \in \Gamma$:
$$\| \pi_n(g)  \xi_n\rho_n(g)^* - \xi_n\|_{2,M_n} / \|q_n \|_{2,M_n}  \to_n 0.$$


\end{proposition}
Before proving the above proposition, let us collect some easy facts that are relevant to passing to large corners of an algebra, which we repeatedly use in the proof of Proposition \ref{prop:rigidity_almost_rep}.
\begin{lemma}\label{lem:flex_technicalities}
Let $(M_n,\tau_n)$ be a sequence of tracial von Neumann algebras, $t_n \in M_n$ be a sequences of projections with $\| t_n \|_{2,M} \to 1$, and $N_n = t_n M_n t_n$ be the corners with respect to $t_n$.
The following hold:
\begin{enumerate}
    \item  if $x_n \in N_n$ is a non-zero sequence, we have 
    $$\| x_n \|_{2,N_n} = \| x_n \|_{2,M_n} / \| t_n \|_{2,M_n} \text{ for all $n$},$$
    so that $\lim_{n\to \infty}\| x_n \|_{N_n} / \|x_n \|_{M_n} = 1$.
    \item If $x_n \in M_n$ is a sequence such that $\sup_n \| x_n \|_{\op} < \infty$, then $\| t_n x_n t_n - x_n\|_{2,M_n} \to_n 0$.
\end{enumerate}
\end{lemma}
\begin{proof}[Proof of Proposition \ref{prop:rigidity_almost_rep}]
    
    Let $\delta > 0$ be given, and choose $\epsilon > 0$ according to Lemma \ref{lem:T_FD_rigidity_rep_strong}. Let $q_n$, $M_n$, $\pi_n: \Gamma \to U(q_nM_nq_n)$, $\rho_n : \Gamma \to U(M_n)$ be given as in the statement of the proposition.
    Note that by $(1)$ in Lemma \ref{lem:flex_technicalities}, the Hilbert--Schmidt norm on $q_n M_n q_n$ is given by $\| x \|_{2, M_n} / \| q_n \|_{2,M_n}$ for $x \in q_n M_n q_n$.
    By applying Hilbert--Schmidt stability to $\rho_n$, we see that there exists finite dimensional factors $\widetilde{M}_n$, and projections $t_n \in \widetilde{M}_n$ such that $t_n \widetilde{M}_n t_n \simeq M_n$, $\| t_n \|_{2, \widetilde{M_n}} \to 1$ and there are representations $\tilde \rho_n: \Gamma \to U(\widetilde{M}_n)$ such that $\| \rho_n(g) - t_n \tilde \rho_n(g) t_n \|_{2, M_n} \to 0$ for all $g \in \Gamma$.
    We can also apply stability to $\pi_n$, and up to possibly enlarging $\widetilde{M}_n$ conclude that there exist projections $\tilde q_n \in \widetilde{M}_n$ such that $q_n \leq \tilde q_n$, $\|q_n\|_{2,\widetilde{M}_n} / \| \tilde q_n \|_{2,\widetilde{M}_n} \to 1$ and there are representations $\tilde{\pi}_n: \Gamma \to U(\tilde q_n \widetilde{M_n} \tilde q_n)$ such that $\| \pi_n(g) - q_n \tilde \pi_n(g) q_n \|_{2, M_n} / \|q_n \|_{2,M_n} \to 0$ for all $g \in \Gamma$.
    
    By  $(2)$ in Lemma \ref{lem:flex_technicalities}, for every $x_n \in \widetilde{M}_n$ of uniformly bounded operator norm, we have $ \| x_n -   t_n x_n t_n \|_{2, \widetilde{M}_n} \to 0$. 
    Consequently, by $(1)$ in Lemma \ref{lem:flex_technicalities}, we have  $\| \rho_n(g) - \tilde \rho_n(g) \|_{2, \widetilde{M_n}} \leq \| \rho_n(g) -  t_n \tilde \rho_n(g) t_n \|_{2, \widetilde{M_n}} + \|  t_n \tilde \rho_n(g) t_n - \tilde \rho_n(g) \|_{2, \widetilde{M_n}} \to 0$ for all $g \in \Gamma$.
    Similarly, for every $x_n \in \tilde q_n \widetilde{M}_n \tilde q_n$ of uniformly bounded operator norm, we have $ \| x_n -   q_n x_n q_n \|_{2, \widetilde{M}_n} / \| \tilde q_n\|_{2, \widetilde{M}_n} \to 0$,  and consequently we have $\| \pi_n(g) - \tilde \pi_n(g) \|_{2, \widetilde{M}_n} / \| \tilde q_n \|_{2,\widetilde{M}_n} \to 0$ for all $g \in \Gamma$.
    We have the following:
    \begin{multline*}
    \Vert q_n \rho_n(g) q_n- \tilde q_n\tilde \rho_n(g)\tilde q_n \|_{2,\widetilde{M}_n} / \|\tilde q_n \|_{2,\widetilde{M}_n}\\
    \leq \|q_n \rho_n(g) q_n - \tilde q_n \rho_n(g) \tilde q_n \|_{2,\widetilde{M}_n} / \| \tilde q_n \|_{2,\widetilde{M}_n} \\
    + \| \tilde q_n(\rho_n(g) - \tilde \rho_n(g))\tilde q_n\|_{2,\widetilde{M}_n} / \| \tilde q_n \|_{2,\widetilde{M}_n}.
    \end{multline*}
    Now, since $\| \tilde q_n(\rho_n(g) - \tilde \rho_n(g))\tilde q_n\|_{2,\widetilde{M}_n} \leq \| \rho_n(g) - \tilde \rho_n(g)\|_{2,\widetilde{M}_n}$,  and since $\liminf_n \|\widetilde{q}_n \|_{2, \widetilde{M}_n} \geq \liminf _n\|q_n \|_{2, \widetilde{M}_n} = \liminf _n\|q_n \|_{2, M_n}  > 0$, together with $(2)$ in Lemma \ref{lem:flex_technicalities} we see that the right hand side converges to $0$.
    Now, we have for all $s \in S$:
    \begin{align*}
        \|\tilde\pi_n(s) -\tilde q_n\tilde \rho_n(s) \tilde q_n\|_{2,\widetilde{M}_n} / \| \tilde q_n\|_{2, \widetilde{M_n} }&\leq  \| \pi_n(s) - \tilde \pi_n(s) \|_{2,\widetilde{M}_n} / \| \tilde q_n \|_{2,\widetilde{M}_n} \\
        &+ \Vert q_n \rho_n(s) q_n- \tilde q_n\tilde \rho_n(s)\tilde q_n \|_{2,\widetilde{M}_n} / \|\tilde q_n \|_{2,\widetilde{M}_n}\\
        &+ \| \pi_n(s) - q_n \rho_n(s) q_n \|_{2,\widetilde{M}_n} / \| \tilde q_n \|_{2,\widetilde{M}_n} 
    \end{align*}
    Using the assumption that $\| \pi_n(s) - q_n\rho_n(s) q_n\|_{2,M_n} / \|q_n \|_{2,M_n}< \epsilon \text{ for all $s \in S$}$ and almost every $n \in \N$, by taking the limsup to both sides, we see that for almost every $n$:
    $$\|\tilde\pi_n(s) -\tilde q_n\tilde \rho_n(s) \tilde q_n\|_{2,\widetilde{M}_n} / \| \tilde q_n\|_{2, \widetilde{M_n} } \leq \epsilon \text{ for every $s \in S$. }$$
    Now we can apply Lemma \ref{lem:T_FD_rigidity_rep_strong} and deduce that there exist $\xi_n \in \tilde q_n \widetilde M_n$ all of operator norm bounded by $2$ with $\| \xi_n - \tilde q_n \|_{2, \widetilde M_n} / \| \tilde q_n \|_{2,\widetilde{M_n}} \leq \delta$ and $\| \xi_n \|_{2,\widetilde{M}_n} / \| \tilde q_n \|_{2,\widetilde{M}_n} = 1$ for every $n$, such that $\tilde \pi_n(g) \xi_n \tilde \rho_n(g) = \xi_n$ for all $g \in \Gamma$.
 Define $\eta_n = q_n \xi_n t_n \in q_n M_n$. Observe that $\eta_n$ is a sequence with operator norm bounded by $2$.

 We claim that $\| \eta_n - q_n \|_{2,M_n} / \| q_n \|_{2,M_n} \leq  2\delta$ and  $\| \eta_n \|_{2, M_n} / \| q_n \|_{2,M_n} \to_n 1$. Indeed, since $\| \xi_n \|_{\op} \leq 2$:
 \begin{align*}
     \| \eta_n - &\xi_n \|_{2, \widetilde{M}_n} / \| \tilde q_n \|_{2,\widetilde{M}_n} \leq \| q_n\xi_n (t_n - 1) \|_{2, \widetilde{M}_n} / \| \tilde q_n \|_{2,\widetilde{M}_n} \\
     &+ \| (q_n - \tilde q_n )\xi_n  \|_{2, \widetilde{M}_n} / \| \tilde q_n \|_{2,\widetilde{M}_n}\\
     &\leq 2(\| 1 - t_n \|_{2,\widetilde{M}_n} + \| \tilde q_n - q_n \|_{2, \widetilde{M}_n}) / \| \tilde q_n \|_{2, \widetilde{M}_n}.
 \end{align*}
 As $\liminf_n \|\widetilde{q}_n \|_{2, \widetilde{M}_n} > 0$, we see the right hand side converges to $0$. As such, by $(1)$ in Lemma \ref{lem:flex_technicalities} and the previous estimates, we get the claim.

 Now, since $\| \eta_n - \xi_n \|_{2, \widetilde{M}_n} / \| \tilde q_n \|_{2,\widetilde{M}_n} \to_n 0$ and for all $g \in \Gamma$, $\tilde \pi_n(g) \xi_n \tilde \rho_n(g) = \xi_n$, we also have $\| \tilde\pi_n(g)  \eta_n\tilde \rho_n(g)^* - \eta_n\|_{2,\widetilde{M}_n} / \|\tilde q_n \|_{2,\widetilde M_n} \to_n 0$. Furthermore:
 \begin{align*}
    \| \pi_n(g)  \eta_n\rho_n(g)^* &- \eta_n\|_{2,\widetilde{M}_n} / \|\tilde q_n \|_{2,\widetilde M_n} \\
    &\leq \| (\pi_n(g) - \tilde \pi_n(g)) \eta_n \rho_n(g)^* \|_{2, \widetilde{M}_n} / \| \tilde q_n \|_{2, \widetilde{M}_n}\\
    &+\| \tilde \pi_n(g) \eta_n (\rho_n(g) - \tilde\rho_n(g))^* \|_{2, \widetilde{M}_n}/ \| \tilde q_n \|_{2, \widetilde{M}_n}\\
    &+ \| \widetilde\pi_n(g)  \eta_n \widetilde\rho_n(g)^* - \eta_n\|_{2,\widetilde{M}_n} / \|\tilde q_n \|_{2,\widetilde M_n} 
 \end{align*}
 The right hand side converges to $0$ by the previous estimates and the fact that $\| \eta_n \|_{\op} \leq 2$.
 Consequently, by $(1)$ in Lemma \ref{lem:flex_technicalities} we have: 
$$\| \pi_n(g)  \eta_n\rho_n(g)^* - \eta_n\|_{2,M_n} / \|q_n \|_{2,M_n}  \to_n 0 \text{ for all $g \in \Gamma$}.$$
By renormalizing $\eta_n$ with $\| q_n \|_{2, M_n} / \| \eta_n \|_{2,M_n}$, making it a unit vector, we obtain the proposition.
\end{proof}

\section{Stability versus hyperlinearity}\label{sec:non-hyperlinear}

In this section we prove Theorem \ref{thm:main-non-hyperlinear} and  Corollary \ref{cor:SL_2}, which will follow from the following theorem regarding general $S$-arithmetic groups.
We refer to \cite{margulis1991discrete, Witte} for background on $S$-arithmetic groups.
\begin{theorem}
    \label{thm:main_S_arith}
    Let $F$ be a totally real number field, $S$ be a finite set of valuations on $F$ including all Archimedean ones, and $\mathcal O_S$ be the ring of $S$-integers in $F$.
    Let $\bG$ be a connected, absolutely almost simple, absolutely simply connected algebraic group defined over $F$. 
    
    Assume that the $F$-rank of $\bG$ is at least $1$, and the $S$-rank of $\bG$ is at least $2$, and that for one of the Archimedean valuations $v_0 \in S$, the Lie group $G(K_{v_0})$ has infinite cyclic fundamental group, where  $K_{v_0} \simeq \R$ is the completion of $F$ with respect to $v_0$.
    Let $\Gamma \leq \prod_{v \in S} \bG(K_v)$ be a lattice commensurable to $\bG(\mathcal O_S)$.
    If $\Gamma$ is Hilbert--Schmidt stable, then $\Gamma$ has a non-hyperlinear finite central extension.
\end{theorem}

We will fix the following notation throughout this section:
Recall that $\omega$ is a nonprincipal ultrafilter on the natural numbers. 
To avoid confusion, we will always add the relevant von Neumann algebra $M$ to the notation of the $2$-norm $\| \cdot \|_{2,M}$.
For a Hilbert space $\cH$, we will also make use of the (unnormalized) Hilbert--Schmidt on $HS(\cH)$, the Hilbert--Schmidt operators on $\cH$. To avoid confusion with the normalized Hilbert--Schmidt norm relevant in the tracial setting, we will refer to this norm as the Frobenius norm and denote it by $\| \cdot \|_{F}$. We also let $\| \cdot\|_{\tr}$ denote the Schatten 1-norm on trace class operators.
Further, for the rest of the paper, we will fix a number $1/2 > \theta_0 > 0$ such that $1 -\sqrt{2\theta_0} \geq 1/\sqrt{2}$, and $1 - \sqrt{2\theta_0} -\sqrt{1-(1-\sqrt{2\theta_0})^2} > \frac{1}{2}$.

\subsection{A Connes type lemma}
We will make use of the following two technical lemmas, the first of which can be compared with \cite[Theorem 1.2.2]{Connes}
   \begin{lemma}
       \label{lem:Connes_trick}
       Let $\cH$ be a Hilbert space, $U_i \in U(\cH)$ a (possibly infinite) collection of unitaries and $T \in HS(\cH) \simeq \cH \otimes \overline{\cH}$ a Hilbert--Schmidt operator such that:
       \begin{enumerate}
           \item $\| U_iTU_i^* - T \|_{F} \leq \epsilon$ for all $i$.
           \item $\| T \|_{F} = 1$.
           \item There is a unit vector $\xi \in \cH$ such that $\| T - \xi \otimes \bar \xi \|_{F} < \theta_0$, where $\xi \otimes \bar \xi \in HS(\cH)$ is the rank one projection on $\spam(\xi)$.
       \end{enumerate}
   Then the top singular value  $\lambda_1$ of  $T$ is of multiplicity $1$ and it satisfies $\lambda_1\geq 1/2$. Furthermore, $|\langle U_i\eta,\eta \rangle| > \sqrt{1- 8\epsilon}$ for all $i$,  where $\eta \in \cH$ is a choice of a unit vector with respect to the $\lambda_1$-singular value. 
   \end{lemma}
   \begin{proof}
       We claim that $\| U_i T^*TU_i^* - T^*T \|_{\tr}< 2\epsilon$ for all $i$. Indeed, we have $\| U_iT^*U_i^* - T^* \|_{F}  < \epsilon$, and so
   
   \begin{align*}
       \|U_i T^*U_i^*U_iT U_i^* - T^*T \|_{\tr} &\leq \| (U_i T^* U_i^* - T^*) U_i T U_i^* \|_{\tr} + \| T^*(U_iTU_i^* - T) \|_{\tr}\\
       &\leq \| (U_i T^* U_i^* - T^*) \|_F\| U_i T U_i^* \|_F\\
       &+ \| T^* \|_F \|(U_iTU_i^* - T) \|_{F} \\
       &< 2\epsilon.
   \end{align*}
       Note that we used the inequality $\|A B\|_{\tr} \leq \|A \|_F\|B\|_F$ for $A,B \in HS(\cH)$. 
       Now, by the Powers-Størmer inequality \cite[Proposition 6.2.4]{BO}, we have
       $$ \| U_i |T|U_i^* - |T| \|_F \leq \| U_i T^*T U_i^* - T^*T\|_{\tr}^{1/2} \leq \sqrt{2\epsilon}.$$
       Denote by $S = |T|$, and note that it is a positive Hilbert--Schmidt operator. Since $\| T -\xi \otimes \bar\xi \|_{F} < \theta_0$ and $(\xi \otimes \bar \xi) = (\xi \otimes \bar \xi)^* = |(\xi \otimes \bar \xi)|$, by a similar analysis to above we get that $\|S - \xi \otimes \bar{\xi} \|_F < \sqrt{2\theta_0}$ . 
       Note that we in particular have $\| U_i S U_i^* - S \|_{\op} < \sqrt{2\epsilon}$,  and $\| S \|_F = 1$.
       Let $\lambda_1$ be the largest eigenvalue of $S$, we have $\lambda_1 = \| S \|_{\op}$. Since $\| S - \xi \otimes \bar\xi \|_{op}  < \sqrt{2\theta_0}$, so $ \lambda_1 > 1 - \sqrt{2\theta_0}$. 
       Moreover $\mathrm{mult}(\lambda_1)$, the multiplicity of  $\lambda_1$, satisfies $\mathrm{mult}(\lambda_1) (1- \sqrt{2\theta_0})^2< \mathrm{mult}(\lambda_1) \lambda_1^2\leq \| S \|_F^2 = 1$, and therefore $\mathrm{mult}(\lambda_1)=1$.

       By the spectral theorem for self-adjoint compact operators, there exists an orthonormal basis $e_n$, so that $$S = \sum_{n=1}^\infty \lambda_n e_n \otimes \bar{e}_n,$$ and $\lambda_1 > \lambda_2 \geq\lambda_3 \dots\geq 0$  is the decreasing  sequence of (possibly repeated) eigenvalues.  Denote $R=\sum_{n=2}^\infty \lambda_ne_n\otimes \bar e_n$. Since $\lambda_1^2> (1-\sqrt{2\theta_0})^2$ and $\|S\|_F^2=1$, we have by Pythagoras that $\| R \|_{\op}^2 \leq \|R\|_F^2=\sum_{n\geq 2} |\lambda_n|^2<1-(1-\sqrt{2\theta_0})^2$.
       Fix $i$, and we will show that $|\langle U_ie_1, e_1\rangle| \geq \sqrt{1-8\epsilon}$, then taking $\eta = e_1$ we will have proved the claim.
       Denote  $v = U_ie_1$, and write $v = \langle v,e_1 \rangle e_1  + w$, where $w \in \overline{\spam\{e_n | n \geq 2\}}$. 
       Then since $\| U_i S U_i^* - S \|_{\op} < \sqrt{2\epsilon}$ we have 
       \begin{align*}
           \sqrt{2\epsilon}&> \|U_iSU_i^*v - Sv \| =\|\lambda_1 v - Rv \|\\
           &= \| \lambda_1  \langle v,e_1 \rangle e_1  + \lambda_1w  - \lambda_1  \langle v,e_1 \rangle e_1  - Sw \| \\
           &= \| \lambda_1w - Rw \|\\
           &\geq \|\lambda_1w\| - \|Rw\| \\
           &\geq (\lambda_1 - \sqrt{1-(1-\sqrt{2\theta_0})^2})\| w\|\\
           &\geq (1 - \sqrt{2\theta_0} -\sqrt{1-(1-\sqrt{2\theta_0})^2}) \|w\|\\
           &\geq (1/2) \| w\|
       \end{align*}
   Where we used that $1 - \sqrt{2\theta_0} -\sqrt{1-(1-\sqrt{2\theta_0})^2} > \frac{1}{2}$ in the last inequality. We thus have that $\| U_ie_1 - \langle U_ie_1,e_1 \rangle e_1\| = \| v - e_1\otimes \bar e_1(v) \|= \|w\| < 2 \sqrt{2\epsilon}$.
   As such, we have that $8\epsilon> \| U_ie_1 - \langle U_ie_1,e_1 \rangle e_1\|^2 = 1 + |\langle U_ie_1,e_1  \rangle |^2 - 2|\langle U_i e_1, e_1 \rangle|^2$.
   Consequently, we have $|\langle U_i e_1, e_1 \rangle| > \sqrt{1- 8\epsilon}$ as promised. 
   \end{proof}

We deduce the following, when specializing to the standard and coarse bimodules of a tracial von Neumann algebra:

\begin{corollary}
\label{cor:Connes_trick_vN}
    Let $(M,\tau)$ be a tracial von Neumann algebra, $p \in M$ a non-zero projection, $u_i \in U(pMp)$ and $v_i \in U(M)$ a (possibly infinite) collection of unitaries. 
    If $\xi \in pM \otimes \overline{pM}$ is an element such that 
    \begin{enumerate}
        \item  for  all $i \leq k$:
        \begin{align*}
            \| (u_i\otimes \bar u_i)\xi  (v_i \otimes \bar v_i)^* &- \xi \|_{2,M\otimes\overline M}/ \|p\|^2_{2,M} \\
            &=\| ((u_i Jv_iJ)\otimes \overline{u_iJv_iJ})\xi - \xi \|_{2,M\otimes\overline M}/ \|p\|^2_{2,M} \leq \epsilon
        \end{align*}
        \item $\| \xi \|_{2,M \otimes \overline{M}} / \|p\|^2_{2,M}= 1$.
        \item $\| \xi - p \otimes \bar p \|_{2, M \otimes \overline M}/ \|p\|^2_{2,M} < \theta_0$. 
    \end{enumerate}
   Then there exists $\eta \in pM$ such that $\| \eta \|_{2,M} / \|p \|_{2,M} = 1$, $\| \eta \|_{op} \leq 2 \| \xi\|_{\op} / \| p \|_{2,M}$ and $|\langle u_i\eta v_i^*,\eta \rangle_{2,M} / \| p \|_{2,M}^2| > \sqrt{1- 8\epsilon}$ for all $i$.
\end{corollary}

\begin{proof}
    This is  an application of Lemma \ref{lem:Connes_trick} to the Hilbert space $\cH =p L^2(M)$ equipped with the norm $\| x\|_{2,M} / \| p \|_{2,M}$ for $x \in \cH$, by noting the canonical isomorphism of Hilbert spaces $HS(\cH) \simeq \cH \otimes \overline \cH \simeq (p \otimes \bar p)(L^2(M) \otimes L^2(\overline M))$, and taking $T \in HS(\cH)$ to be the operator corresponding to $\xi$, and the unitaries $U_i = u_i \cdot J v_i J \in U(\cH)$.
    Since under this isomorphism, we have $(U_i\otimes \overline{U}_i) \xi$ corresponds to the operator $U_iTU_i^* \in HS(\cH)$, the first three conditions above correspond to the three conditions listed in Lemma \ref{lem:Connes_trick}, so we obtain a vector $\eta \in \cH$ such that $\| \eta \|_{2,M} / \|p \|_{2,M} = 1$ and $|\langle u_i \eta v_i^*, \eta \rangle / \| p \|_{2,M}^2 |= |\langle U_i\eta,\eta \rangle_\cH| > \sqrt{1- 8\epsilon}$ for all $i$. 
    
    We are left with proving $\| \eta \|_{op} \leq 2 \| \xi\|_{\op} / \| p \|_{2,M}$.
    Using the singular value decomposition of operator $T \in HS(\cH)$ corresponding to $\xi$, we can write $\xi = \sum_{n=1}^\infty \lambda_n\cdot  (f_n \otimes \bar e_n)$, where $e_n, f_n \in \cH$ are orthonormal (also known as the Schmidt decomposition of $\xi$) and $\lambda_n \geq 0$ decreasing. As such, $\langle f_n, f_m\rangle_{2,M} = \| p\|_{2,M}^2\cdot \delta_{n,m}$ and $\langle e_n, e_m\rangle_{2,M} = \| p\|_{2,M}^2\cdot \delta_{n,m}$.
    Note that $|T| = \sum_{n=1}^\infty \lambda_n(e_n\otimes \bar e_n)$, so $\lambda_1$  and  $e_1 = \eta$, are the same as in Lemma \ref{lem:Connes_trick}. 
    In particular  $\lambda_1 \geq  1/2$.
    Write $e_1=  | e_1| u^*$, the polar decomposition of $e_1 \in \cH$, where $u \in U(M)$ is unitary\footnote{Note that for tracial von Neumann algebra, the unitary part of the polar decomposition can be taken to be unitary, and not just a partial isometry \cite{943650}.}.
  %
    Let $x \in M$ be such $\| x \|_{2,M}  = 1$ and $\langle |e_1| \cdot x, x \rangle_{2,M}  \geq \||e_1|\|_{\op}/2 = \| \eta \|_{\op} / 2$. Then, since $\|1\otimes \overline{ux} \|_{2,M\otimes \overline M}   = 1,  \| f_1 \otimes \bar x \|_{2, M \otimes \overline{M}}   = \| p \|_{2, M}$, we have:
    \begin{align*}
        \|\xi \|_{\op} \| p \|_{2,M}&\geq\langle\xi \cdot (1\otimes \overline{ux}), f_1 \otimes \bar x \rangle_{2, M \otimes \overline M} \\
        &= \sum_{n=1}^\infty \lambda_n\langle (f_n  \otimes \overline{e_n ux}), f_1 \otimes \bar x \rangle_{2, M \otimes \overline M} \\
        &= \sum_{n=1}^\infty \lambda_n\langle f_n, f_1 \rangle_{2, M}\cdot  \overline{\langle e_n ux,  x \rangle_{2, M }} \\
        &= \lambda_1 \| p \|_{2,M}^2 \overline{\langle e_1 ux,  x \rangle_{2, M }} \\
        &\geq 1/2 \| p \|_{2,M}^2 \langle |e_1| x,  x \rangle_{2, M } \\
        &\geq 1/4 \| p \|_{2,M}^2 \|\eta\|_{\op}\\
    \end{align*}
    and so we have $\| \eta \|_{\op} \leq 4 \|\xi \|_{\op} / \| p \|_{2,M}$.

\end{proof}

\subsection{Asymptotically projective representations}

We denote by $\T$ the circle group, and by $Z^2(\Gamma, \T)$ the abelian group of $2$-cocycles $c: \Gamma \times \Gamma\to \T$, where $\T$ is treated as a $\Gamma$-module with a trivial action.

\begin{definition}
    \label{def:asym_proj}
    A sequence of maps $\pi_n: \Gamma \to U(d_n)$ is said to be an \emph{asymptotically projective representation} if there is a $2$-cocycle $c \in Z^2(\Gamma, \T)$ such that for all $g,h \in \Gamma$:
    $$\lim_n \| \pi_n(g)\pi_n(h) - c(g,h) \pi_n(gh)\|_2 = 0.$$
    We will say that the cocycle $c$ is \emph{associated} with $\pi_n$.
\end{definition}

Note that if $c$ is trivial, then $\pi_n$ above is an asymptotic representation.
We will show the following local rigidity property for asymptotically projective representations whose cocycle is not a coboundary.
This can be as thought of as a version of lemma due to Nicoara, Popa and Sasyk \cite[Lemma 1.1]{zbMATH05116360} suitable to deal with asymptotically projective representations.
    
\begin{theorem}
\label{thm:NPS_analog}
Assume $\Gamma = \langle S \rangle$ has property (T;FD) and is Hilbert--Schmidt stable.

There exists $\epsilon> 0$ such that for every sequence $d_n \in \N$, $M_n = M_{d_n}(\C)$ and projections $q_n \in M_n$ with $\liminf_n \| q_n\|_{2,M_n} > 0$, the following holds: If  $\rho_n: \Gamma \to U( M_n )$ is an asymptotic representation and $\pi_n:\Gamma\rightarrow U(q_n M_n q_n)$ is an asymptotically projective representation  (with the corresponding normalized Hilbert Schmidt norms) such that for almost every $n$:
$$\| \pi_n(s) - q_n\rho_n(s) q_n\|_{2,M_n} / \|q_n \|_{2,M_n}< \epsilon \text{ for all $s \in S$,}$$
then the cocycle associated with $\pi_n$ is a $2$-coboundary.
\end{theorem}

    \begin{proof}
        Pick $\epsilon$ according $\delta := \theta_0$ in Proposition \ref{prop:rigidity_almost_rep}. 
        Let $\pi_n:\Gamma \to U(q_nM_nq_n)$, $\rho_n: \Gamma \to U(M_n)$ be given as in the statement of the theorem, with $\| \pi_n(s) - q_n\rho_n(s) q_n\|_{2,M_n} / \|q_n \|_{2,M_n}< \epsilon / 2$ for $s \in S$. Let $c\in Z^2(\Gamma, \T)$ be the cocycle associated with $\pi_n$, so that:
    $$\lim_n \| \pi_n(g)\pi_n(h) - c(g,h) \pi_n(gh)\|_{2,M_n} / \| q_n \|_{2,M_n} = 0 \text{ for all $g,h \in \Gamma$.}$$
    
        Now we can consider $\widetilde{M}_n = M_n \otimes \overline{M}_n$, $\tilde q_n = q_n\otimes \bar q_n$,  $\tilde \pi_n = \pi_n\otimes \overline{\pi_n}: \Gamma \to U(\tilde q_n \cdot \widetilde M_{n} \cdot \tilde q_n)$ and $\tilde \rho_n = \rho_n \otimes \overline{\rho_n}: \Gamma \to U(\widetilde M_{n})$. We claim that $\tilde \pi_n, \tilde \rho_n$ are asymptotic representations (with respect to the normalized Hilbert-Schmidt norms on $\tilde q_n \widetilde M_n \tilde q_n$, $\widetilde M_n$ respectively).
        Indeed, we have for all $g,h \in \Gamma$:
        \begin{align*}
             \| \tilde\pi_n(g)\tilde \pi_n(h) &-  \tilde\pi_n(gh)\|_{2, \widetilde{M}_n} / \| \tilde q_n \|_{2,\widetilde{M}_n} \\
             &\leq   \|(\pi_n(g) \pi_n(h) - c(g,h)\pi_n(gh))\otimes \overline{\pi_n(g) \pi_n(h)} \|_{2,\widetilde{M}_n} / \| \tilde q_n \|_{2,\widetilde{M}_n}\\
             &+   \| \pi_n(gh)\otimes \overline{(c(g,h)^{-1}\pi_n(g) \pi_n(h) - \pi_n(gh))} \|_{2,\widetilde{M}_n} / \| \tilde q_n \|_{2,\widetilde{M}_n}\\
             &\leq 2\|\pi_n(g) \pi_n(h) - c(g,h)\pi_n(gh) \|_{2, M_n} / \| q_n \|_{2,M_n}
        \end{align*}
        Since the right most bound converges to $0$, we see that $\tilde \pi_n$ is an asymptotic representation. The same calculation (with the trivial cocycle) shows that $\tilde \rho_n$ is an asymptotic representation.
        Further, note for all $s \in S$:
        \begin{align*}
            \| \tilde \pi_n(s) -& \tilde q_n \tilde \rho_n(s) \tilde q_n\|_{2, \widetilde M_n} / \| \tilde q_n \|_{2,\widetilde M_n} \\
            &\leq \| (\pi_n(s) - q_n \rho_n(s) q_n) \otimes \overline{\pi_n(s)} \|_{2,\widetilde{M}_n} / \| \tilde q_n\|_{2, \widetilde{M}_n} \\
            &+ \| (q_n \rho_n(s) q_n) \otimes \overline{(\pi_n(s) -q_n \rho_n(s) q_n)} \|_{2,\widetilde{M}_n} / \| \tilde q_n\|_{2, \widetilde{M}_n} \\
            &\leq2\| \pi_n(s) - q_n\rho_n(s) q_n\|_{2,M_n} / \|q_n \|_{2,M_n},
        \end{align*}
        the right most term is smaller than $\epsilon$ for almost every $n$. As such, by Proposition \ref{prop:rigidity_almost_rep}  there exists unit vectors $\xi_n \in \tilde q_n \widetilde{M}_n$ of uniformly bounded operator norm such that $\| \xi_n - \tilde q_n \|_{2, \widetilde{M}_n} / \| \tilde q_n\|_{2, \widetilde{M}_n} < \delta$, and$\|\tilde \pi_n(g)\xi_n \tilde \rho_n(g)^* - \xi_n \|_{2,\widetilde M_n} / \| \tilde q_n \|_{2,\widetilde M_n} \to_{n} 0$ for all $g \in \Gamma$.
           Consequently, for each $n$ large enough we can use Corollary \ref{cor:Connes_trick_vN} (with $u_i, v_i$ being the unitaries $\pi_n(g)$, $\rho_n(g)$, $g\in \Gamma$ respectively)  to find a sequence of unit vectors $\eta_n \in q_nM_n$, for which $\| \eta_n \|_{\op} \leq 2 \|\xi_n \|_{\op} / \| q_n \|_{2,M_n}$ and for all $g \in \Gamma$:
           $$|\langle \pi_n(g) \eta_n\rho_n(g)^*, \eta_n \rangle_{2,M_n} / \|q_n \|_{2,M_n}^2|  \to _n 1.$$ 
           Note that since $\liminf_{n} \| q_n \|_{2,M_n} > 0$, we also have that $\eta_n$ is uniformly bounded in operator norm.
        For every $g \in \Gamma$, set $\lambda_n(g) \in \T$ to  be the unique scalar such that such that $\langle \pi_n(g) \eta_n\rho_n(g)^*, \lambda_n(g)\eta_n \rangle_{2,M_n} > 0$.
        Then for every $g \in \Gamma$, $\| \pi_n(g)\eta_n \rho_n(g)^* - \lambda_n(g)\eta_n \|_{2,M_n} / \| q_n\|_{2,M_n} \to_n 0$.
        Let $C$ be a uniform bound on the operator norm of $\eta_n$,  and note that we have:
        \begin{align*}
            \| c(g,h)\pi_n(g)\pi_n(h)&\eta_n \rho_n(h)^*\rho_n(g)^* - \pi_n(gh) \eta_n \rho_n(gh)^*\|_{2,M_n} \\
            & \leq \|(c(g,h)\pi_n(g)\pi_n(h) - \pi_n(gh)) \eta_n \rho_n(h)^*\rho_n(g)^* \|_{2,M_n} \\
            & \; + \| \pi_n(gh) \eta_n (\rho_n(h)^*\rho_n(g)^* -\rho_n(gh)^*) \|_{2,M_n}\\
            & \leq C\|c(g,h)\pi_n(g)\pi_n(h) - \pi_n(gh)\|_{2,M_n} \\
            & \; + C\|  \rho_n(h)^*\rho_n(g)^* -\rho_n(gh)^* \|_{2,M_n}.
        \end{align*}
        After normalizing the above inequality by $\| q_n \|_{2,M_n}$, we see the rightmost term converges to $0$. Thus we  have $\| c(g,h)\pi_n(g)\pi_n(h)\eta_n \rho_n(h)^*\rho_n(g)^* - \lambda_n(gh) \eta_n \|_{2,M_n} / \| q_n \|_{2,M_n} \to_n 0$.
        Further, we have for all $g,h \in \Gamma$:
        \begin{align*}
             \| c(g,h)\lambda_n(h) &\lambda_n(g)\eta_n -\lambda_n(gh)\eta_n\|_{2,M_n}\\
             &\leq \| c(g,h)\lambda_n(h) (\pi_n(g) \eta_n \rho_n(g)^* - \lambda_n(g)\eta_n)\|_{2,M_n}\\ 
             &\;+ \| c(g,h) (\pi_n(g) (\pi_n(h) \eta_n \rho_n(h)^* - \lambda_n(h)\eta_n) \rho_n(g)^*)\|_{2,M_n}\\ 
            &+ \| c(g,h)\pi_n(g)\pi_n(h)\eta_n \rho_n(h)^*\rho_n(g)^* -\lambda_n(gh)\eta_n\|_{2,M_n}\\
            &\leq \|\pi_n(g) \eta_n \rho_n(g)^* - \lambda_n(g)\eta_n \|_{2,M_n}\\
            &\;+\|\pi_n(h) \eta_n \rho_n(h)^* - \lambda_n(h)\eta_n \|_{2,M_n}\\
            &+ \| c(g,h)\pi_n(g)\pi_n(h)\eta_n \rho_n(h)^*\rho_n(g)^* -\lambda_n(gh)\eta_n\|_{2,M_n}.\\
        \end{align*} 
        Again, after normalizing the inequality by $\| q_n \|_{2,M_n}$, we see the rightmost term converges to $0$. Since $\| \eta_n \|_{2,M_n} / \| q_n\|_{2,M_n} = 1$, we have that for every $g,h \in \Gamma$, $\lim_n \lambda_n(g)^{-1}\lambda_n(h)^{-1} \lambda_n(gh) = c(g,h)$. 
        Since $\T^\Gamma$ is compact, we can take a partial pointwise limit $\lambda: \Gamma \to \T$ of $\lambda_n$.
        It then follows that $\lambda(g) \lambda(h) \lambda(gh)^{-1} = c(g,h)$ for all $g,h \in \Gamma$, so $c$ is a $2$-coboundary.
        \end{proof}
\begin{remark}\label{rem:asymp_rig_ult}
    In the statement of Theorem \ref{thm:NPS_analog}, Definition \ref{def:asym_proj} and the propositions that lead up to it, we could have replaced all limits (and statements regarding almost every $n \in \N$) with ultralimits with respect $\omega$ (and statements regarding $\omega$-almost every $n$), and the proofs would remain the same.
\end{remark}

For a countable group $\Gamma$ and a $2$-cocycle $c: \Gamma \times \Gamma \to \T$, one associates a \emph{twisted group von Neumann algebra} $L_c(\Gamma)$ defined as follows:
For each $g \in \Gamma$ let $u_g \in U(\ell^2(\Gamma))$ be the unitary defined by $u_g \cdot \delta_h = c(g,h) \delta_{gh}$.
It then follows that $g \mapsto u_g$ is a projective unitary representation with associated cocycle $c$, that is $u_g u_h = c(g,h) u_{gh}$ for all $g,h \in \Gamma$. We then set $L_c(\Gamma) = \{ u_g | \; g\in \Gamma\}'' \subset B(\ell^2(\Gamma))$, see \cite[Section 2.1]{Dog} for more on these algebras. 

We can now prove the following Theorem analogous to \cite[Theorem 3.1]{Dog}, the crucial difference being that we do not assume property (T), but merely property (T;FD).
\begin{theorem} \label{thm:Ioana_analog}
Let $\Gamma$ be a finitely generated group with property (T;FD). Assume there exists a sequence of $2$-cocycles 
$c_n \in Z^2(\Gamma, \T)$ with the following properties:

\begin{enumerate}
    \item For each $n$, the cohomology class $[c_n] \in H^2(\Gamma, \T)$ is nontrivial.
    \item For any $g,h \in \Gamma$, $c_n(g,h) \xrightarrow[n \to \infty]{} 1$.
    \item For every $n$, the twisted group von Neumann algebra $L_{c_n}\Gamma$ is Connes embeddable.
\end{enumerate}
Then $\Gamma$ is not Hilbert--Schmidt stable.
\end{theorem}

\begin{proof}
Assume by contradiction that $\Gamma = \langle S\rangle$, $S = \{ s_1, \dots, s_k\}$ is Hilbert--Schmidt stable. As $\Gamma$ has asymptotic property (T;FD), we can choose $\epsilon$ as promised in Theorem \ref{thm:NPS_analog}.
Choose a presentation $\Gamma = \langle S \vert R \rangle$, $R \subset \mathbb F_{S}$ a set of words in the free group on $S$.
For $g \in \Gamma$ and $n \in \mathbb N$, denote by $u_{g,n} \in L_{c_n} \Gamma$ the corresponding canonical unitary that satisfies:

$$ u_{g,n} u_{h,n} = c_n(g,h) u_{gh, n} \text{ for all $g,h \in \Gamma$} $$
By assumption, for every $n \in \mathbb N$, there is a trace-preserving embedding of $L_{c_n} \Gamma$ in $M = \prod_{m \to \omega} M_m(\C)$. 
To simplify notation, we fix such an embedding and simply think of $L_{c_n} \Gamma \subset \mathcal M$ sitting as a subalgebra.

For fixed $g$ and $n$, we choose a representative sequence $u_{g,n} = \{ u_{g,n} (m) \}_{m}$ where $u_{g,n}(m) \in U(m)$
(we can always do so, as unitaries in the ultraproduct lift to unitaries in the product, see for example \cite[Lemma 2.2]{zbMATH07469105}). 
Note that for a fixed $n$, the maps $\{ g\in \Gamma \mapsto u_{g,n}(m) \}_m$ are an asymptotically projective representation of $\Gamma$ (as in definition \ref{def:asym_proj}, see also Remark \ref{rem:asymp_rig_ult}), whose associated cocycle is $c_n$.
Indeed, for every $g,h \in \Gamma$, we have:
\begin{align*}
 \ultlim{m}&\| u_{g,n}(m) u_{h,n}(m) - c_n(g,h) u_{gh, n}(m) \|_{2,M_m(\C)} \\
 &= \| u_{g,n} u_{h,n} - c_n(g,h) u_{gh, n}\|_{2,M} =0.
\end{align*}
By Hilbert--Schmidt stability, there exists $\delta > 0$, and a finite set $R_0 \subset R$
such that if $\varphi: S \to U(m)$ satisfies $ \Vert r(\varphi(s_1), \dots, \varphi(s_k)) - 1 \Vert_{2, M_m(\C)} < \delta \text{ for all $r \in R_0$}$, 
then there exists a representation $\rho: \Gamma \to U(2m)$ such that $ \Vert \varphi(s_i) - q_m \rho(s_i) q_m \Vert_{2, M_m(\C)} < \epsilon$ for all $i$, 
where $q_m \in M_{2m}(\C)$ is the projection onto the first $m$ coordinates.
Let us denote by $M_m = M_{2m}(\C)$, so that $\| q_m \|_{2, M_m} = 1/2$ for all $m$, in particular, $\liminf \|q_m \|_{2,M_m} > 0$.

Denote by $L$ the maximal length of a word appearing in $R_0$. Choose and fix $n$ large enough such that 
$\vert c_n(g, h) - 1 \vert < \frac{\delta}{2(L+1)}$ for all $g,h$ in the ball of radius $L$ in the Cayley graph $\text{Cay}(\Gamma, S)$.
As a result, we have for all $r \in R_0$:
\begin{align*}
 \Vert r(u_{s_1,n}, \dots, u_{s_k, n}) - 1_{\mathcal M} \Vert_{2, \mathcal M} &= \Vert C u_{r(s_1, \dots, s_k), n} - 1_{\mathcal M} \Vert_{2, \mathcal M}\\
 &= \Vert C \cdot 1_{\mathcal M} - 1_{\mathcal M} \Vert_{2, \mathcal M} < \frac{\delta}{2},
\end{align*}
where $C$ is a product of $L+1$ instances of $c_n(g,h)$ for some $g,h$ in the ball of radius $L$ in $\text{Cay}(\Gamma, S)$. Thus, for $\omega$-almost every $m$, we have

$$ \Vert r(u_{s_1,n}(m), \dots, u_{s_k, n}(m)) - 1 \Vert_{2, M_m(\C)} < \delta.$$
As such, we can choose representations $\rho_{n,m}: \Gamma \to U(M_m)$ such that for $\omega$-almost every $m \in \N$:
\begin{align*}
    \Vert q_m \rho_{n,m}(s_i) q_m - u_{s_i, n}(m) \Vert_{2, M_m} / \| q_m\|_{2,M_m} < \epsilon 
\end{align*}
Since $\{ g\in \Gamma \mapsto u_{g,n}(m) \}_m$ is an asymptotically projective representation, and $\rho_{n,m}$ are representations, we conclude by Theorem \ref{thm:NPS_analog}  (see also remark \ref{rem:asymp_rig_ult}) that the cocycle $c_n$ is in fact a $2$-coboundary, contradicting our assumptions.
\end{proof}

\subsection{Hyperlinearity of central extensions and asymptotically projective representations}

To finish our proof, we will use the following Lemma of A. Thom \cite[Lemma 3.3]{Thom} connecting the hyperlinearity of central extensions to Connes embeddability of the arising twisted group von Neumann algebras.

\begin{lemma}[cf. Lemma 3.3 in \cite{Thom}] \label{lem:Thom}
Let $A$ be a countable abelian group, $\Gamma$ a countable group. Let  $\alpha \in Z^2(\Gamma, A)$ a cocycle, and consider the  central extension arising from it
\begin{center}
\begin{tikzcd}
0 \arrow[r] & A \arrow[r] & \widetilde{\Gamma} \arrow[r] & \Gamma \arrow[r] & 0.
\end{tikzcd}
\end{center}
Then $\widetilde \Gamma$ is hyperlinear iff for every character $\chi \in \widehat{A}$ the twisted group von Neumann algebra $L_{\chi \circ \alpha}\Gamma$ is Connes embeddable.
\end{lemma}

We will also need the following coefficient change result:

\begin{proposition}[cf. Proposition 7.2 in \cite{Dog}] \label{prop:coeff_change}
Let $\Gamma$ be a countable group. Let $c \in Z^2(\Gamma, \mathbb Z)$ be a cocycle and consider the central extension arising from it
\begin{center}
\begin{tikzcd}
1 \arrow[r] & \mathbb Z \arrow[r] & \widetilde{\Gamma} \arrow[r] & \Gamma \arrow[r] & 1.
\end{tikzcd}
\end{center}
If $\widetilde{\Gamma}$ has finite abelianization, then there exists a character $\chi \in \widehat{\mathbb Z}$ such that
$\chi \circ c \in Z^2(\Gamma, \mathbb{T})$ is not a coboundary.
\end{proposition}

Note that in \cite{Dog}, the above proposition appears with an assumption of property (T) for $\widetilde{\Gamma}$.
However, it is easy to see that the proof only uses the fact that a countable property (T) group has finite abelianization.
We can now finally prove Theorem \ref{thm:main-non-hyperlinear} with the relevant tools in place.

\begin{proof}[Proof of Theorem \ref{thm:main-non-hyperlinear}]
Let us denote by $\alpha \in Z^2(\Gamma, \Z)$ the cocycle corresponding to the given central extension $\widetilde{\Gamma}$.
Assume that for all $N \in \N$, the group $\widetilde{\Gamma}_N=\widetilde{\Gamma} / (N\cdot \Z)$ is hyperlinear.
Since hyperlinearity is closed for projective limits of groups \cite[Corollary 7.5.9]{CS-C}, it follows that $\widetilde{\Gamma}$ is hyperlinear as well.
Thus, by Lemma \ref{lem:Thom} we have for each $\chi \in \widehat{\Z}$ that $L_{\chi \circ \alpha} \Gamma$ is 
Connes embeddable. Since $\widetilde{\Gamma}$ is assumed to have finite abelianization, by Proposition \ref{prop:coeff_change} there exits some $\chi$ such that
$[\chi \circ \alpha] \neq 0 \in H^2(\Gamma, \T)$. Note that $\{ \chi \vert \, [\chi \circ \alpha] = 0 \in H^2(\Gamma,\T)\}$ is a subgroup of $\widehat \Z$, as a result, it can not contain an open neighborhood of the identity  ($\widehat \Z$ is connected).
Thus, there exists a sequence $\chi_n \in \widehat \Z$ such that $\chi_n \xrightarrow[n \to \infty]{} e \in \widehat \Z$ but
$[\chi_n \circ \alpha] \neq 0$ for every $n$. Take $c_n = \chi_n \circ \alpha$ we obtain a sequence satisfying the
conditions of Theorem \ref{thm:Ioana_analog}. Consequently, $\Gamma$ is not Hilbert--Schmidt-stable.
\end{proof}

\begin{proof}[Proof of Theorem \ref{thm:main_S_arith}]
Let $F$, $S$, $\bG$, $v_0$, $\Gamma$ be as in the statement of the theorem.
It is well known that since $\bG$ is $F$-anisotropic, $\Gamma$ has the congruence subgroup property and property (T;FD) (see the proof of \cite[Theorem 9.5]{LZ}).
Recall that for the non-compact simple Lie group $\bG(F_{v_0})$ we have $\pi_1(\bG(F_{v_0})) = \Z$, so that the universal cover group $\widetilde{\bG(F_{v_0})}$ has center $\Z$.
Let $G_0 = \prod_{v \in S} \bG(F_v)$ and $G = \prod_{v \in S \setminus \{v_0\}} \bG(F_{v_0}) \times \widetilde{\bG(F_{v_0})}$, and let $p: G \to G_0$ denote the natural covering map.
Note that $\ker p = \Z$, and let $\widetilde{\Gamma} = p^{-1}(\Gamma)$.
It is known that $\widetilde{\Gamma}$ has finite abelianization (see for example \cite[Proposition 2.1]{DH}), and it is not residually finite (see \cite[Theorem 3.3 and Remark 3.5]{stover2024residual}, or \cite{Rag_torsion}).
Consequently, $\widetilde{\Gamma}$ fits the conditions of Theorem \ref{thm:main-non-hyperlinear}.

\end{proof}
\section{Hyperfiniteness and Schramm's theorem for marked von Neumann algebras}\label{sec:hyperfinite}

In this section we introduce hyperfiniteness for sequences of marked von Neumann algebras, as well as formulate and prove an suitable analog of a fundamental theorem of O. Schramm  \cite{Schramm}.
Our treatment is heavily influenced by the theory of graph limits, for which the reader can consult the wonderful book \cite{zbMATH06122804}.
Fix a number $k \in \N$, recall that a sequence of (undirected) finite graphs $G_n = (V_n, E_n)$ with vertex degrees at most $k$ is said to be \emph{hyperfinite}  if for every $\epsilon > 0$ there exists $D$ such that for all $n$ large enough, there exists $X \subset V_n$ with  $|X | / |V_n | \geq 1- \epsilon$ such that the connected components in the induced subgraph $X\subset V_n$ are all of size at most $D$, and (see \cite{Elek}).
Hyperfiniteness can be regarded as a strong negation of being expander graphs, see \cite[Chapter 21]{zbMATH06122804} for more information.
We start with defining our suitable analog for the largest size of a connected component.

\begin{definition}[Degree of subhomogeneity]
If $Q$ is a finite dimensional von Neumann algebra, that is, a direct sum of matrix algebras.
we let $\mathfrak c(Q)$ be the maximal $D$ such that there exists a factor in the decomposition of $Q$ isomorphic to $M_D(\C)$.
\end{definition}

Note that $\mathfrak c(Q) \leq D$ for a type I-von Neumann algebra $Q$ is equivalent to $Q$ being \emph{$D$-subhomogeneous}, meaning that all irreducible representations of $Q$ are of dimension at most $D$.
From now on, fix $k \in \N$ (which we think of as a bound on the degree of our "graphs"). 
We can now define two notions of hyperfiniteness, the difference between them being whether the "size" of the subalgebra $Q$ is measured by $\fc(Q)$ or $\dim(Q)$. 

\begin{definition}[Hyperfinite Tuples]\label{def:hyperfinite tuple}
Let $x_1, \dots, x_k \in M$ be a tuple in a tracial von Neumann algebra $M$.
Given $\epsilon>0, D \in \N$, the tuple $(x_1, \dots, x_k)$ is said to be \textbf{\emph{$(\epsilon,D)$-hyperfinite}} if  there exists a unital finite dimensional $*$-subalgebra $Q \leq  M$ with $\fc(Q) \leq D$ and $\| x_i - \mathbb E_Q(x_i) \|_2 \leq \epsilon$ for all $i \leq k$.

    A sequence of marked von Neumann algebras $( M_n; x_n^1, \dots x_n^k)$ is said to be \emph{hyperfinite} if for every $\epsilon > 0$ there exists $D,N \in \N$ such that for all $n \geq N$, the tuple $(x_n^1, \dots, x_n^k)$ is $(\epsilon, D)$-hyperfinite.
\end{definition}

\begin{definition}[Strongly Hyperfinite Tuples]\label{def:strong hyperfinite tuple}
Let $x_1, \dots, x_k \in M$ be a tuple in a tracial von Neumann algebra $M$.
Given $\epsilon>0, D \in \N$, the tuple $(x_1, \dots, x_k)$ is said to be \textbf{\emph{$(\epsilon,D)$-strongly hyperfinite}} if  there exists a unital finite dimensional $*$-subalgebra $Q \leq  M$ with $\dim(Q) \leq D$ and $\| x_i - \mathbb E_Q(x_i) \|_2 \leq \epsilon$ for all $i \leq k$.

    A sequence of tuples in tracial von Neumann algebras $( M_n; x_n^1, \dots x_n^k)$ is said to be \emph{strongly-hyperfinite} if for every $\epsilon > 0$ there exists $D,N \in \N$ such that for all $n \geq N$, the tuple $(x_n^1, \dots, x_n^k)$ is $(\epsilon, D)$-strongly hyperfinite.
\end{definition}

Note that for the above definitions, we did not assume the tuples $x_n^1, \dots x_n^k$ generate $M_n$. 
As a consequence of the following theorem, we see that in fact hyperfiniteness can be detected in the algebras generated by $x_n^1, \dots, x_n^k$, so one does not loose much generality by working only with marked von Neumann algebras.
It is also easy to check that for a single marked von Neumann algebra $(M; \bar x)$, hyperfiniteness of the constant sequence $(M; \bar x)$ is the same as $M$ being hyperfinite in the usual sense \cite[Definition 11.1.2]{Popa}. 
Let us take this opportunity to mention Connes' foundational result \cite{Connes} showing the equivalence between hyperfiniteness and \emph{amenability} for von Neumann algebras.

A celebrated theorem of Schramm is that \emph{hyperfiniteness of a bounded degree graph sequence can be detected by a Benjamini-Schramm limit \cite{Schramm}}, here we prove the analogous statement for limits of sequences of marked von Neumann algebras.
We will also show that for sequences, hyperfiniteness and strong hyperfiniteness are in fact the same notion. 

\begin{theorem} \label{thm:vN_Schramm}
Fix $k \in \N$, and let $(\mathcal M_n; x_n^1, \dots ,x_n^k)$ be a sequence of tuples in tracial von Neumann algebras $(M_n, \tau_n)$, and assume that the marked von Neumann algebras $(\langle x_n^1, \dots , x_n^k \rangle''; x_n^1, \dots ,x_n^k)$ converge to a marked von Neumann algebra $(\mathcal M; x^1, \dots, x^k)$. The following are equivalent:
\begin{enumerate}
    \item The sequence $(\mathcal M_n; x_n^1, \dots ,x_n^k)$ is hyperfinite.
    \item The sequence $(\mathcal M_n; x_n^1, \dots ,x_n^k)$ is strongly hyperfinite.
    \item $\mathcal M$ is hyperfinite (equivalently, amenable).
\end{enumerate}
\end{theorem}

Our proof will make essential use of the tracial ultraproduct.
We first record the two following folklore lemmas:

\begin{lemma}\label{lem:almost_hyperfinite_is_hyperfinite}
    Let $M \leq \cN$ be an inclusion of tracial von Neumann algebras. Assume $M = \langle x_1, \dots, x_k \rangle$ and that for every $\epsilon >0$, there exists some hyperfinite subalgebra $\mathcal S \leq \cN$ such that $\| x_i - \bE _\mathcal S(x_i) \|_2 \leq \epsilon$ for all $i \leq k$. Then $M$ is hyperfinite.
\end{lemma}
\begin{proof}
    Recall that a von Neumann algebra $M$ is \emph{semidiscrete} if for every $\epsilon >0$ and finite subset $F \subset M$ there exist normal ucp maps $\theta: M \to Q$ and $\xi: Q \to M$, for some finite dimensional von Neumann algebra $Q$, such that $\| x - \xi \circ \theta (x) \|_2 \leq \epsilon$ for all $x \in F$.
    It is known  that semidiscretness is equivalent to hyperfiniteness \cite{Connes}, so it is enough to show that the $M$ is semidiscrete.
    Given $\epsilon>0$, there exists some hyperfinite $\mathcal S \leq \cN$ such that $d_2(x_i, \mathcal S) \leq \epsilon/2$ for all $i \leq k$. 
    We can find $Q \leq \mathcal S$ such that $d_2(\bE_\mathcal S(x_i),  Q) \leq \epsilon $  for all $i$. It is then easy to see that $\| x_i - \bE_\mathcal M (\bE_Q(x_i))\|_2 \leq 2\epsilon$, so taking the ucp maps to be the conditional expectations, we are done.
\end{proof}

\begin{lemma}\label{lem:lifting_matrix_alg}
    Let $M_n$ be a sequence of tracial von Neumann algebras, and $Q$ be a (unital) finite dimensional subalgebra of $\ultprod{n} M_n$.
    Then there exists a sequence of (\emph{possibly non-unital}) $*$-homomorphisms $\pi_n: Q \to M_n$ such that $(\pi_n(x))_n = x$ for all $x \in Q$, as elements of $\ultprod{n} M_n$.
\end{lemma}
\begin{proof}
    This follows from \cite[Lemma 1.1.12]{Popa}.
    Indeed, write $Q = \sum_{k=1}^m Q^k$, where $Q^k \simeq M_{d_k}(\C)$. By choosing matrix units, we can find a collection of partial isometries $e_i^k \in Q^k, i\leq d_k$ and a projection $q^k_1\in Q^k$ satisfying $(e_i^k)^* e_j^k  =\delta_{i,j} \cdot q^k_1 $, such that  $q^k :=\sum_{i=1}^{d_k} e_i^k (e_i^k)^*$ is the unit of $Q^k$ and  $Q^k = \spam \{e^k_i(e_j^k)^*| i, j \leq d_k \}$.
    Let $q^k_i = e_i^k (e_i^k)^*$. Note that $q^k_i$, $k \leq m, i \leq d_k$ is a finite set of disjoint projections summing up to $1$, so we can lift it to collections of disjoint  projections $q^k_i(n) \in M_n$, $k\leq m,i \leq d_k, n \in \N$ with $\sum_{k,i} q^k_i(n) = 1$ for all $n$, such that $q^k_i = (q^k_i(n))_n$ as elements of the ultraproduct (see for example \cite{HS_alg}).
     For each $k \leq m$, set $q^k(n) = \sum_i q_i^k(n)$, and note that we have  $q^k = (q^k(n))_n$ as elements of the ultraproduct, and $Q^k = q^k Q q^k \leq \ultprod{n} q^k(n) M_n q^k(n)$.
    
    Now, by applying \cite[Lemma 11.1.12]{Popa} for each $k$ separately, we deduce that for almost every $n$ there exist  partial isometries $e^k_i(n) \in q^k(n)M_n q^k(n)$ with $e_i^k(n)^* e_j^k(n) = \delta_{i,j} \cdot p^k(n)$ for all $i,j,k,n$, where $p^k(n) \in q^k(n)M_n q^k(n)$ are projections such that $p^k(n) \leq q^k_1(n)$ and $\ultlim{n} \|q^k_1(n) -  p^k(n) \|_2 =0$ and additionally, $(e^k_i(n))_n = e_i^k$ as elements of the ultraproduct, for all $k \leq m, i \leq d_k$.
    Since $\ultlim{n} \|q^k_1(n) -  p^k(n) \|_2 =0$, in fact $(p^k(n))_n = q^k_1$.
    It now follows that the map $\pi_n(e_i^k) = e_i^k(n)$ extends to a (possibly non-unital) $*$-homomorphism $\pi_n: Q \to M_n$ for almost every $n$, yielding the desired conclusion.
\end{proof}
\begin{proof}[Proof of Theorem \ref{thm:vN_Schramm}]
    Let $\mathcal N = \prod_{n\to \omega} \mathcal M_n$ be the tracial ultraproduct, as in section \ref{sec:preliminaries}.
    Note that the elements $\widetilde{x}^i = (x_n^i)_n \in \mathcal N, i \leq k$ generate a marked von Neumann algebra which is isomorphic to $(M; x^1, \dots, x^k)$. Indeed, for any $*$-monomial $m(X_1, \dots, X_k)$, by convergence of the tuples $(M_n; x_n^1, \dots, x_n^k)$ we have:
    \begin{align*}
        \tau_\cN (m(\widetilde{x}^1, \dots, \widetilde{x}^k)) = \lim_{n \to \omega} \tau_{M_n} (m(x_n^1, \dots, x_n^k) ) = \tau_{M} (m(x^1, \dots, x^k))
    \end{align*}
    so the assignment $x^i \mapsto \widetilde{x}^i$ extends to a trace preserving $*$-isomorphism of $M$ and $\langle \widetilde{x}^1, \dots, \widetilde{x}^k \rangle$.
    Thus, we can assume that $x^i= (x_n^i)_n$ for each $i$ to begin with, and $M  \leq \mathcal N$ is generated by them.

    Since obviously $(2) \implies (1)$, it is enough to show $(1) \implies (3)$ and $(3) \implies (2)$.
        We first prove $(1) \implies (3)$, so we assume that the sequence $M_n$ is hyperfinite.
        Fix $\epsilon > 0$. There exists $D, N$ such that for all $n \geq N$ there exists finite dimensional $*$-subalgebras $Q_n \leq M_n$ such that $\fc(Q_n) \leq D$ and $d_2(x_n^i, Q_n) \leq \epsilon$ for all $i \leq k$.
        Since $\fc(Q_n) \leq D$, there exists some finite dimensional abelian algebras $A_n^d$ for $1\leq d \leq D$ such that $Q_n \simeq \bigoplus_{d=1}^D \bM_d(\C) \otimes A^d_n$ for all $n \geq N$. Now consider $Q = \prod_{N \leq n \to \omega} Q_n \leq \mathcal N$.
        Since $d_2(x_n^i, Q_n) \leq \epsilon$ for all $N \leq n$ and $i \leq k$, we see that $d(x^i, Q) \leq \epsilon$ for all $i \leq k$ (this follows easily from the fact that $d_2(x^i, Q) = \| x^i - \mathbb E_Q(x^i)\|_2$). 
        We claim that the von Neumann algebra $Q$ is a hyperfinite.
        Indeed, using the fact that taking ultraproducts commutes with taking finite direct sums:
        $$Q = \prod_{N \leq n \to \omega} \bigoplus_{d=1}^D \bM_d(\C) \otimes A_n^d  = \bigoplus_{d=1}^D \prod_{N \leq n \to \omega}  \bM_d(\C) \otimes A_n^d = \bigoplus_{d=1}^D  \bM_d(\C) \otimes A^d$$
        Where $A^d = \prod_{N \leq n \to \omega} A_n^d$. Since $A^d$  is abelian for all $d \leq D$,  $Q$ is hyperfinite. 
        Since $\epsilon$ was arbitrary, it then follows that $M$ is hyperfinite by Lemma \ref{lem:almost_hyperfinite_is_hyperfinite}.\footnote{We thank David Gao and Adrian Ioana for this part of the proof.}

        We now prove $(3) \implies (2)$, so we assume that $\mathcal M$ is hyperfinite. Assume by contradiction that $\mathcal M_n$ is not strongly hyperfinite. 
        By passing to a subsequence (which still converges to $\mathcal M$) and relabeling indices, we may assume that there exists $\epsilon >0$ such that for all $n \in \N$, $(M_n; x^1_n, \dots, x^k_n)$ is not $(\epsilon, n)$-strongly hyperfinite.
        Since $\mathcal M$ is hyperfinite, there exists a (unital) finite dimensional subalgebra $Q \leq \mathcal M \leq \mathcal N$, such that $Q = \langle \bE_Q(x^1), \dots, \bE_Q(x^d) \rangle$ and $d_2(x^i, Q) \leq \epsilon/4$ for all $i$. 
        By Lemma \ref{lem:lifting_matrix_alg}, there exist  (possibly non-unital) $*$-homomorphisms $\pi_n: Q \to  \mathcal M_n$ such that for all $x \in Q$, $x = (\pi_n(x))_n$ (as elements of the ultraproduct).
        Let $q_n = \pi_n(1) \in M_n$ be the units of $\pi_n(Q)$, and note that $\ultlim{n} \| q_n \|_2 = 1$.
        Let $\tilde{Q}_n = \pi_n(Q) \oplus \C(1-q_n) \leq M_n$,  observe that $\tilde{Q}_n $ is a unital $*$-subalgebra of $M_n$ with $\dim (\tilde{Q}_n ) \leq \dim(Q) + 1$.
        Consequently, for each $i \leq k$, we have
        $$\lim_{n \to \omega} \| x^i_n - \pi_n(\bE_Q(x^i)) \|_2 =  \| x^i - \bE_Q(x^i) \|_2 \leq \epsilon/4$$ 
        So by definition of an ultralimit, there exists some $n \geq \dim(Q) + 1$ such that for all $i \leq k$, $\| x^i_n - \pi_n(\bE_Q(x^i)) \| \leq \epsilon/2$, so that in particular $d_2(x_n^i, \tilde{Q}_n ) \leq \epsilon/2$.
        Since $\dim(Q) + 1 \leq n$, this is a contradiction to the assumption that $(M; x_n^1, \dots, x_n^k)$ is not $(\epsilon, n)$-hyperfinite.
        This concludes the proof.
        
\end{proof}

\section{Hyperfinite Hilbert--Schmidt stability} \label{sec:hyp_HS_stab}

Let us begin by defining hyperfinite Hilbert--Schmidt stability, using the newly introduced terminology of the previous section.

\begin{definition}\label{def:hyp_HS_stab}
    Let $\Gamma = \langle s_1, \dots, s_k \rangle$ be a finitely generated group $\Gamma$.
    An asymptotic representation $\pi_n :\Gamma \to U(d_n)$ is said to be \emph{hyperfinite} if the sequence of tuples $(\pi_n(s_1), \dots, \pi_n(s_k))$ is hyperfinite.
    \
    Finally, $\Gamma$ is said to be \emph{hyperfinitely Hilbert--Schmidt stable} if for every asymptotic representation $\pi_n : \Gamma \to U(d_n)$ which is hyperfinite, there exists a sequence of representations $\rho_n: \Gamma \to U(d_n)$ such that $\| \pi_n(g) - \rho_n(g) \|_2 \to 0$ for every $g \in \Gamma$.
\end{definition}

It is a routine check to prove that the Definition given here does not depend on the generating set (see for example \cite{Dog_Thesis}).
As a starting point, we give the following characterization of hyperfinite Hilbert--Schmidt stability in terms of characters, which is a direct generalization of Hadwin and Shulman's Theorem \cite{HS_grp}.

\begin{theorem}\label{thm:HS_for_non_amenable_grps}
Let $\Gamma$ be a finitely generated group, then the following are equivalent:
\begin{enumerate}
    \item $\Gamma$ is hyperfinitely Hilbert--Schmidt stable.
    \item Every von Neumann amenable character of  $\Gamma$ is a limit of finite dimensional traces. 
\end{enumerate}
\end{theorem}
\begin{proof}
Assume $(1)$, and let $\varphi$ be a von Neumann amenable character. Then the von Neumann algebra of $\varphi$ is in particular Connes embeddable. It follows that $\varphi=\lim_n \frac{1}{d_n}\mathrm{tr}\circ \pi_n$ for an asymptotic representation $\pi_n:\Gamma\to U(d_n)$. Fix a finite generating set $\gamma_1 ,...,\gamma_m\in \Gamma$. The tuples $(\pi_n (\gamma_1 ),...,(\pi_n(\gamma)))$ converges to $(\pi(\gamma_1),...,\pi(\gamma_m))$ where $\pi$ is the representation corresponding to $\varphi$. Since $\pi(\Gamma)''$ is amenable, it is moreover hyperfinite by Connes theorem. 
It follows from Theorem \ref{thm:vN_Schramm} that the asymptotic representation $\pi_n$ is hyperfinite. Thus, as $\Gamma$ hyperfinitely Hilbert--Schmidt stable, there exists a sequence of genuine representations $\rho_n:\Gamma\to U(d_n)$ with $\|\pi_n(\gamma)-\rho_n(\gamma)\|\to 0$ for all $\gamma\in \Gamma$. It follows that $\varphi$ is a limit of finite dimensional traces
$$\varphi=\lim_n \frac{1}{d_n}\mathrm{tr}\circ \pi_n=\lim_n \frac{1}{d_n}\mathrm{tr}\circ \rho_n $$

The argument for the other direction is a combination of Theorem \ref{thm:vN_Schramm}  with the argument given in \cite[Theorem 4]{HS_grp}, as follows. 
    Suppose $\pi_n: \Gamma \to U(d_n)$ is a hyperfinite asymptotic representation of $\Gamma$. It is enough to show that any subsequence of $\pi_n$ admits a further subsequence which can be corrected to a genuine sequence of representation. Thus, consider a subsequence of $\pi_{n_m}$, and let $(M_m;\bar{x}_m)$ be the corresponding marked von Neumann algebras. By  compactness, we may extract a subsequence  $(M_{m_l};\bar{x}_{m_l})$ which converges to a limit $(M,\bar{x})$. By Theorem \ref{thm:vN_Schramm}, $M$ is amenable. Moreover, since $\pi_{n_{m_l}}$ are asymptotic representations, the map $\pi:\Gamma\to U(M)$ corresponding to the marked von Neumann algebra $(M,\bar{x})$ is a genuine representation. 
    It follows that $\varphi=\tau_M \circ \pi$ is a trace, and by assumption, it is a limit of of finite dimensional traces $\varphi_l=\frac{1}{d_l'}\mathrm{tr}\circ \rho_l$ for some $d_l' \in \N $ and some representations $\rho_l:\Gamma\to U(d_l' )$. 
    
    While there is no reason to expect that $d_{l'} = d_{n_{m_l}}$, by replacing $\rho_l$  by a direct sum of $\rho_l$ with itself, and  adding the trivial representation to if  necessary, we get a new representation whose trace also converges to $\varphi$, and whose dimension is the same as of $\pi_{n_{m_l}}$ (see \cite[Lemma 3.7]{HS_alg}). 
    We are thus in the situation that $ \rho_l$ and $\pi_{n_{m_l}}$ are of equal dimensions and they converge to the same limiting trace $\varphi$, which generates a hyperfinite algebra. We can now apply \cite[Theorem 1.1]{HS_alg}, to find unitaries $u_l \in U(d_{n_{m_l}})$ such that $\pi_{n_{m_l}}$ and the conjugates of the form $\mathrm{Ad}_{u_l}\circ \rho_l$ satisfy $\| \pi_{n_{m_l}}(g) - \mathrm{Ad}_{u_l}(\rho_l(g)) \|_2 \to_l 0$.
    Then  $\mathrm{Ad}_{u_l}\circ \rho_l$  are the desired representations, thus proving hyperfinite HS-stability. 
\end{proof} 

An immediate consequence of the above theorem is that an amenable group is hyperfinitely Hilbert--Schmidt stable if and only if it is Hilbert--Schmidt stable in the usual sense.
This is because all traces on an amenable group are von-Neumann amenable, so that condition $(2)$ in Theorem \ref{thm:HS_for_non_amenable_grps} is equivalent to the criterion of Hadwin and Shulman for Hilbert--Schmidt stability of an amenable group.
As a nice bonus, we get that some permanence properties of Hilbert--Schmidt stability of amenable groups, which are not true in general for non-amenable groups, hold for hyperfinite Hilbert--Schmidt stability as well.

\begin{proposition}\label{prop:hyp_stab_prod}
    If $\Gamma_1, \Gamma_2$ are hyperfinitely Hilbert--Schmidt stable, then so is $\Gamma_1 \times\Gamma_2$.
\end{proposition}
\begin{proof}
    Let $\varphi$ be a von Neumann amenable character of $\Gamma_1 \times \Gamma_2$. By a result of Thoma \cite{thoma1964unitare}, there exists characters $\varphi_1 \in \Ch{\Gamma_1}, \varphi_2 \in \Ch{\Gamma_2}$ such that $\varphi = \varphi_1 \otimes \varphi_2$, so that $\pi_{\varphi}(g h) = \pi_{\varphi_1}(g) \otimes \pi_{\varphi_2}(h)$ for $g \in \Gamma_1, h \in \Gamma_2$ and $M_{\varphi} = M_{\varphi_1} \otimes M_{\varphi_2}$.
    Since $M_{\varphi}$ is amenable, so are $M_{\varphi_1}, M_{\varphi_2}$. As such, $\varphi_1 ,\varphi_2$ are von Neumann amenable characters. 
    By Theorem \ref{thm:HS_for_non_amenable_grps}, there are sequence of finite dimensional traces $\varphi_{n,1} \in \Tr{\Gamma_1}, \varphi_{n,2} \in \Tr{\Gamma_2}$ such that $\lim_n \varphi_{n,1}= \varphi_1$ and $ \lim_n \varphi_{n,2} = \varphi_2$.
    It is then easy to see that $\varphi_{n,1} \otimes \varphi_{n,2}$ is a sequence of finite dimensional traces of $\Gamma_1 \times\Gamma_2$ which converge to  $\varphi$. 
    Consequently, $\Gamma_1 \times\Gamma_2$ is hyperfinitely Hilbert--Schmidt stable by Theorem \ref{thm:HS_for_non_amenable_grps}.
    \end{proof}

A particular consequence of the previous proposition is that $\mathbb F_2 \times \mathbb F_2$ is hyperfinitely Hilbert--Schmidt stable, while it is not Hilbert--Schmidt stable, as shown by Ioana \cite{ioana2024almost}.

\section{Robust spectral gap for asymptotic representations}\label{sec:rob}

The proof of Theorem \ref{thm:main-non-hyperlinear} heavily relied on Hilbert--Schmidt stability to provide rigidity properties of asymptotic representations for groups with property (T;FD). 
It is natural to wonder whether it is possible to give forms of rigidity, or spectral gap for asymptotic representations such groups, without the need for stability.
To that end, we develop a notion, which we call \emph{robust property (T;FD)}. 
We will also discuss a robust version of property (T), when the asymptotic representations are considered in \emph{operator norm}.
We will show that both these properties are implied by property (T).

A positive bounded operator \(x \in \mathcal{B}(\mathcal{H})\) is said to have \emph{\(\lambda\)-spectral gap} (for some \(\lambda > 0\)) if
\[
\sigma(x)\,\cap\,(0,\lambda) \;=\; \emptyset,
\]
where $\sigma(x)\subset [0,\infty)$ is the spectrum of $x$.
Given a finite symmetric generating set $S\subset \Gamma$, we denote the corresponding (normalized) \emph{combinatorial Laplacian} by \[\Delta_S = 1 - \frac{1}{|S|}\sum_{s \in S} s \in \C[\Gamma].\]

Given a map  $\pi:\Gamma \to U(\cH)$ that satisfies $\pi(s^{-1}) = \pi(s)^*$ for each $s \in S$, we consider the non-negative  contracting operator  \[\pi(\Delta_S) = 1 - \frac{1}{|S|}\sum_{s\in S} \pi(s)\in \mathcal{B}(\mathcal{H}).\]

Property (T) and its relative variants are typically formulated in terms of the trivial representation being isolated in the Fell topology. We shall use an equivalent definition and its generalizations (see \cite{LZimm}).

\begin{definition}
    Let \(\Gamma\) be a finitely generated group, and let \(\mathcal{Q}\) be a class of von Neumann algebras.
    We say that \(\Gamma\) satisfies \emph{property \((T;\mathcal{Q})\)} if there exist (equivalently, for any) a finite generating set \(S \subset \Gamma\),  and \(\lambda > 0\), such that for every representation $\pi$ for which $\pi(\Gamma)''$ is in $\mathcal Q$, the positive operator \(\pi(\Delta_S)\) has a \(\lambda\)-spectral gap.
\end{definition} 

Consider:
\begin{itemize}
  \item \(\mathrm{FD}\) is the class of finite dimensional von Neumann algebras;
  \item \(\mathrm{CE}\) is the class of Connes embeddable von Neumann algebras;
  \item \(\mathrm{W}^*\) is the class of all tracial von Neumann algebras.
\end{itemize}
Note that  $  \mathrm{FD} \;\subset\; \mathrm{CE} \;\subset\; \mathrm{W}^*$.
We now work towards robust versions of the above notions.

\begin{definition}
    A positive element \(x\) in a tracial von Neumann algebra \((M,\tau)\) is said to have an \emph{\((\epsilon,\alpha)\)-almost \(\lambda\)-spectral gap}, for  \(\epsilon, \alpha, \lambda >  0\), if
\[
\tau\bigl(\chi_{[\alpha,\lambda-\alpha]}(x)\bigr) \;\le\; \epsilon,
\]
that is,  the spectral projection of $x$ corresponding to the interval $[\alpha,\lambda-\alpha]$ is of trace at most $\epsilon$. 
\end{definition}

We can now give a characterization of spectral gap for a limit in the space of marked von Neumann algebras (see \S\ref{subsec:marked}), similarly to how we characterized hyperfiniteness of such a limit in Theorem \ref{thm:vN_Schramm}. 
This in particular implies that uniform spectral gap passes to limits, generalizing \cite{levit2023spectral}.

\begin{proposition}\label{prop:spec_gap_pass_to_limit}
Let $(M_n;\bar x_n)$ be a sequence of $k$-marked von Neumann algebras converging to $(M;\bar x)$, and let $p$ be a $*$-polynomial in $k$ variables, such that for all $n$, $p(\bar{x}_n) \in M_n$ and $p(\bar x)\in M$ are positive.

Then, given $\lambda > 0$, the operator $p(\bar x)$ has $\lambda$-spectral gap if and only if for every $\alpha > 0, \epsilon > 0$, the operators $p(\bar x_n)$ have $(\epsilon, \alpha)$-almost $\lambda$-spectral gap for almost every $n$.
\end{proposition}

\begin{proof}
    By Lemma \ref{lem:convergence of spectral measures} we have that the spectral measures $\mu_n$ of $p(\bar x_n)$ converge weak-$*$ to  the spectral measure $\mu$ of $p(\bar x)$.
    Fix $\lambda > 0$, and we will show that $\mu((0,\lambda))=0$  if and only if
    \begin{equation}\label{eq:conv}
        \text{for every $\alpha > 0$,} \qquad \lim_n \mu_n([\alpha, \lambda - \alpha])  = 0.
    \end{equation}

    If  (\ref{eq:conv}) holds, then Portmanteau's theorem on weak-$*$ convergence implies that
    \begin{align*}
        0\leq \mu((\alpha,\lambda-\alpha)) \leq  \liminf_{n\to \infty} \mu_n((\alpha, \lambda-\alpha))
        \leq \liminf_{n\to\infty} \mu_n([\alpha/2, \lambda-\alpha/2]) = 0.
    \end{align*}
    Since this holds for all $\alpha>0$, it follows that $\mu((0,\lambda))=0$.
    Conversely, if $\mu((0,\lambda))=0$, then again by Portmanteau's theorem we have that for all $\alpha >0$  $$0 \leq \limsup_{n\to \infty} \mu_n([\alpha, \lambda - \alpha]) \leq \mu([\alpha, \lambda - \alpha]) = 0,$$ so that condition (\ref{eq:conv}) holds.
\end{proof}

We can now introduce the robust version of property (T;$\mathcal Q$).

\begin{definition}
Given $\delta > 0$ and a finite subset $S\subset F \subset \Gamma$, a map $\pi: \Gamma \to U(M)$ into  a tracial von Neumann algebra $(M,\tau)$ is said to be a \emph{$(F, \delta)$-almost representation} if:
$$ \| \pi(gh) - \pi(g) \pi(h) \|_{2} \leq \delta\; \text{ for all $g,h \in F$,}$$
and $\pi(s^{-1}) = \pi(s)^{*}$ for all $s \in S$.
\end{definition}

\begin{definition}\label{def:robust}
    Let $\Gamma$ be a finitely presented group, and let \(\mathcal{Q}\) be a collection of tracial von Neumann algebras.
    We say that \(\Gamma \) satisfies \emph{property \((T;\mathcal{Q})_\mathrm{rob}\)} if there exist a finite generating set \(S \subset \Gamma \) and \(\lambda > 0\) such that the following holds:

    For every \(\epsilon > 0, \alpha>0\) there exists \(\delta > 0\) and a finite subset $S \subset F \subset\Gamma$ such that for every \((M,\tau) \in \mathcal{Q}\), and for every \((F,\delta)\)-almost representation \(\pi \colon \Gamma  \to U(M)\), the positive operator \(\pi(\Delta_S) \) has an \((\epsilon, \alpha)\)-almost \(\lambda\)-spectral gap.

\end{definition}

This definition is independent of the generating set. In fact, similarly to Theorem \ref{thm:vN_Schramm} characterizing hyperfiniteness using ultraproducts, we can also characterize robust property $(T;\mathcal Q)_{\mathrm{rob}}$ using ultraproducts. 

\begin{proposition}\label{prop:eq defs of robust}
    Let $\Gamma = \langle S \rangle$ be a finitely generated group, $S=\{s_1, \dots, s_k\}$ a symmetric finite generating set. 
    Fix a collection of \(\mathcal{Q}\) of tracial von Neumann algebras, and let $\mathcal P$ be the collection of all tracial von Neumann algebras that can be embedded in tracial ultraproducts of sequences from $\cQ$.
    The following are equivalent:
    \begin{enumerate}
        \item  $\Gamma$ has property $(T;\cQ)_\mathrm{rob}$.
        \item There exists $\lambda > 0$ such that for every sequence  $(M_n,\tau_n) \in \cQ$ and every  asymptotic representation $\pi_n: \Gamma \to U(M_n)$, for every $\alpha> 0$ we have:
            $$\tau_n(\chi_{[\alpha,\lambda-\alpha]}(\pi_n(\Delta_S))) \to_n 0.$$
        \item  $\Gamma$ has property (T;$\cP$). 
        \item For every $\epsilon> 0 $ there exists $\delta > 0$ such that for every trace $\varphi \in \Tr{\Gamma}$ with $M_\varphi = \pi_\varphi(\Gamma)'' \in \cP$, if $\xi \in L^2(M_\varphi)$ is a unit vector such that $\| \pi_\varphi(g)\xi - \xi \|_2 \leq \delta$ for all $g \in S$, then there exists a $\Gamma$-invariant unit vector $\eta \in L^2(M_\varphi)$ such that $\| \xi - \eta \|_2 \leq \epsilon$. 
    \end{enumerate}
\end{proposition}
\begin{proof}
    The equivalence of $(1)$ and $(2)$ is straightforward, and we leave it to the reader. 
    The equivalence of $(3)$ and $(4)$ follows from the standard characterization of spectral gap for Laplacians, see \cite[Proposition I]{dlHRV} and \cite[Proposition 1.1.9]{BdlV}.
    
    We now prove $(2) \implies (3)$. 
    Fix $\lambda$ as in Definition \ref{def:robust}.
    Let $\varphi \in \Tr{\Gamma}$ be a trace such that $M_\varphi = \pi_\varphi(\Gamma) ''$ is in $\cP$.
    That is, there is a sequence $(M_n,\tau_n) \in \cQ$ such that $M_\varphi \leq \ultprod{n} M_n$.
    In this case, there exists an asymptotic representation $\pi_n: \Gamma \to U(M_n)$ such that the tuples $(\pi_n(s_1), \dots , \pi_n(s_k))_n$ converge in moments to the tuples $(\pi_\varphi(s_1),\dots ,\pi_\varphi(s_k))$ (see \cite[Theorem 7]{OzawaConnes}).
    Further, we can modify $\pi_n$ so that that $\pi_n(s_i^{-1}) = \pi_n(s_i)^{-1}$ for each $i$, as is done in \cite[Section 3.1]{DGLT}.
    As such, by $(2)$ it follows that for every $\alpha> 0$ we have $\tau_n(\chi_{[\alpha,\lambda-\alpha]}(\pi_n(\Delta_S))) \to_n 0$. Using Proposition \ref{prop:spec_gap_pass_to_limit}, we see that $\pi_{\varphi}(\Delta_S)$ has $\lambda$-spectral gap.
    
    Let us now prove $(3) \implies (2)$.
    Fix $\lambda > 0$ which satisfies the condition in $(3)$, and assume by contradiction that $\Gamma$ does not satisfy $(2)$. 
    As such, there is an asymptotic representation $\pi_n: \Gamma \to U(M_n)$, for some $(M_n,\tau_n) \in \cQ$ and $0 < \alpha<\lambda$ such that $\liminf_{n \to \infty} \tau_n(\chi_{[\alpha,\lambda-\alpha]}(\pi_n(\Delta_{S}))) > 0 $.
    By compactness of the space of $k$-marked von Neumann algebras (see Section \ref{subsec:marked}),  there is a subsequence of the tuples $((\pi_n(s_1), \dots , \pi_n(s_k))_n$ that converge in moments to a tuple $(\pi_\varphi({s}_1), \dots, \pi_{\varphi}({s}_k))$, where $\varphi$ is a trace on the free group $\mathbb F_S$ and $\pi_\varphi: \F_S \to U(M_\varphi)$ is the GNS-construction.
    By replacing $\pi_n$ with this subsequence, and simply assume convergence holds.
    Now, since $\pi_n$ is an asymptotic representation,  it follows that in fact $\varphi \in \Tr{\Gamma} \subset \Tr{\F_S}$, so that $\pi_\varphi$ in fact factors to a $\Gamma$-representation $\pi_\varphi: \Gamma \to U(M_\varphi)$ . 
    Further, as in the proof of Theorem \ref{thm:vN_Schramm}, we see that $M_\varphi$ embeds in $\ultprod{n} M_n$, so $M_\varphi \in \cP$.
    Consequently, by the condition in $(3)$ we  have $\mu((0,\lambda)) = \tau(\chi_{(0,\lambda)}(\pi_\varphi(\Delta_S))) = 0$.
    Together with convergence and Proposition \ref{prop:spec_gap_pass_to_limit}, we see that $\lim_{n \to \infty} \tau_n(\chi_{[\alpha,\lambda-\alpha]}(\pi_n(\Delta_{S}))) = 0$, contradiction.
\end{proof}

As a consequence, we obtain the following.
\begin{corollary}\label{cor:eq defs of robust}
The following implications holds:
\begin{align*}
        (T)\Longrightarrow (T;\mathrm{W}^*)_{\mathrm{rob}}&\iff (T;\mathrm{W}^*)\Longrightarrow(T;\mathrm{CE})\\ &\iff (T;\mathrm{FD})_{\mathrm{rob}}
        \Longrightarrow (T;\mathrm{FD})\Longrightarrow (\tau).
\end{align*}
\end{corollary}
\begin{remark}\label{rem:JM}
    It is worthwhile to note that in general, property (T;FD$)_{\mathrm{rob}}$ is \emph{not} implied by property (T;FD). 
    Indeed, if $\Gamma$ is a finitely generated, infinite, simple amenable group (as in \cite{JM}), then it  satisfies property (T;FD) simply because it admits no non-trivial finite dimensional representations.
    However, being amenable, its regular representation generates a Connes embeddable algebra $L(\Gamma)$, while having almost invariant vectors in $L^2(\Gamma)$. This implies that $\Gamma$ does not have (T;FD$)_{\mathrm{rob}}$ by $(4)$ in Proposition \ref{prop:eq defs of robust}.
    It remains an interesting question to decide whether (T;FD$)_{\mathrm{rob}}$ and (T;FD) are different for \emph{residually finite} groups. \end{remark}


\subsection{Property (T) and its robust variant}\label{appendix:Bader}

In this section we record a short proof, due to U. Bader, that approximate representations of Kazhdan groups with respect to the \emph{operator norm} possess a stronger form of approximate spectral gap than the one we discuss in the main body of the paper.
This generalizes the work of \cite{Manuilov2007OnAR}, which only worked under the assumption that Zuk's criterion holds.

\begin{definition}
Given $\delta > 0$ and $F \subset \Gamma$ finite, $\cH$ a Hilbert space, a map $\pi: \Gamma \to U(\mathcal H)$ is said to be an \emph{operator norm $(F, \epsilon)$-approximate representation} if
$$ \| \pi(gh) - \pi(g) \pi(h) \|_{\op} \leq \epsilon\; \text{ for all $g,h \in F$,}$$
and $\pi(g^{-1}) = \pi(g)^{-1}$ for all $g \in F$.
\end{definition}

With the connection to the topic of group stability being apparent, we remark that the measure of defect for the approximate representation is in \emph{operator norm}.

Let $\Gamma$ be a finitely generated group, and fix a symmetric finite generating set $S$ for $\Gamma$.
Given an $(F,\epsilon)$-representation $\pi$ of $\Gamma$, the linear extension $\tilde{\pi}: \mathbb R[\Gamma] \to B(\mathcal H)$ defined by $\tilde{\pi}(\sum_i a_i g_i) = \sum_i a_i \pi(g_i)$
is a linear $*$-preserving map such that for all $\xi, \eta \in \mathbb R[\Gamma]$ with $\supp(\eta), \supp(\xi) \subset F_\lambda$ we have
$$\| \tilde{\pi}(\xi)\tilde{\pi}(\eta) - \tilde{\pi}(\xi\eta) \|_{\op} \leq \epsilon \cdot \| \xi \|_1 \| \eta \|_1.$$
Indeed, linearity and $*$-preservation are immediate. For $\xi, \eta \in \mathbb R[\Gamma]$ supported on $F_\lambda$, write $\xi = \sum_{g \in F_\lambda} a_g g$ and $\eta = \sum_{g \in F_\lambda} b_g g$ , then by the triangular inequality:
\begin{equation*}
    \| \tilde{\pi}(\xi) \tilde{\pi}(\eta) - \tilde{\pi}(\xi  \eta) \|_{\op} \leq \sum_{g,h \in F_\lambda} |a_g b_h| \cdot \Vert \pi(g) \pi(h) - \pi(gh) \|_{\op} \leq \epsilon \| \xi \|_1 \| \eta \|_1. 
\end{equation*}

Consider the corresponding Laplacian $\Delta_S$ considered as an element of the real group ring $\mathbb R[\Gamma]$,

\begin{theorem}[Almost spectral gap for almost representations] \label{thm:op_main}
Let $\Gamma$ be a property (T) group. Let $S$ be a symmetric finite generating set for $\Gamma$, and let
$\kappa > 0$ be a Kazhdan constant for $\Gamma, S$, and let $\lambda = \kappa^2/(2|S|)$.
Then there exists $S \subset F \subset \Gamma$ finite and a constant $C$ such that for every $\epsilon > 0$ and any $(F,\epsilon)$-approximate
representation $\pi: \Gamma \to U(\mathcal H)$, where $\cH$ is a Hilbert space, we have 
$$ \sigma(\tilde{\pi}(\Delta_S)) \subset [0, C \epsilon] \cup [\lambda - C \epsilon, 2], $$
where $\tilde{\pi}(\Delta_S) = 1 - \frac{1}{|S|} \sum_{s \in S} \pi(s) \in B(\mathcal H)$ is the Laplacian associated with $\pi$.
\end{theorem}
\begin{proof}
    
Let $\Delta_S$ be the  Laplacian. By Ozawa's characterization of property (T) \cite{OzawaT}, there are group algebra elements $\xi_1,...,\xi_n\in \R[\Gamma]$ such that
$$ \Delta_S^2 - \lambda \Delta_S = \sum_{i=1}^n \xi_i^* \xi_i$$
We let $F = \bigcup_{i=1}^n \supp(\xi_i)$, and $M = \max \{ \| \xi_i \|_{\infty} \vert 1 \leq i \leq n \}$.
Note that $F \subset \Gamma$ is a finite subset.

Fix an $(F, \epsilon)$-approximate representation $\pi: \Gamma \to U(\mathcal H)$. Then 
$$\tilde{\pi}(\Delta_S^2 - \lambda\Delta_S) = \tilde{\pi}(\Delta_S^2) - \lambda \tilde{\pi}(\Delta_S) \approx_{M^2 \epsilon} \tilde{\pi}(\Delta_S)^2 - \lambda\tilde{\pi}(\Delta_S) = \tilde{\pi}(\Delta_S)^2 - \lambda \tilde{\pi}(\Delta_S).$$
where $a \approx_c b$ means that $\|a-b\|_{\op}\leq c$.
On the other hand:
$$ \tilde{\pi}(\Delta_S^2 - \lambda\Delta_S) = \tilde{\pi}(\sum_{i=1}^n \xi_i^* \xi_i) = \sum_{i=1}^n \tilde{\pi}(\xi_i^* \xi_i) \approx_{n M^2 \epsilon} \sum_{i=1}^n \tilde{\pi}(\xi_i)^* \tilde{\pi}(\xi_i)$$
As a result,  $\tilde{\pi}(\Delta_S)^2 - \lambda \tilde{\pi}(\Delta_S)$  is $C \epsilon$-close in operator norm to a sum of squares, where $C = (n+1)M^2$.

 We claim there is a constant $C'$ independent of $\pi$ such that:
    $$\sigma(\tilde{\pi}(\Delta_S)) \subset [0, C'\epsilon] \cup [\lambda - C'\epsilon, 2]$$
    Note $\frac{1}{|S|} \sum \pi(s)$ is a contraction, so the spectrum of $\tilde{\pi}(\Delta_S)$ is contained in $[0,2]$. As $\tilde{\pi}(\Delta_S)^2 - \lambda \tilde{\pi}(\Delta_S)$  is $C \epsilon$-close to a sum of squares, and sums of squares are non-negative operators, we have $\sigma(\tilde{\pi}(\Delta_S)^2 - \lambda \tilde{\pi}(\Delta_S)) \subset [-C\epsilon, \infty)$ .
    For the parabola $y(x) = x^2 - \lambda x$, we have a minimum at $\frac{\lambda}{2}$ with value $-\frac{\lambda^2}{4}$, and the roots are at $0, \lambda$. 
    Thus, by the spectral mapping theorem, for the spectrum inclusion above to hold we must have $\sigma(\tilde{\pi}(\Delta_S)) \subset [0, C' \epsilon] \cup [\lambda - C' \epsilon, 2]$, for $C' = \frac{2C}{\lambda}$. 
\end{proof}
\begin{remark}
 We note that this proof has the advantage of being effective. Indeed, it guarantees that $\epsilon$-almost representations have a ``$C\epsilon$-almost spectral gap''. One can also obtain a Hilbert–Schmidt version of this, thus making the implication
$(T)\;\Rightarrow\;(T;W^*)_{\mathrm{rob}}$ from Corollary \ref{cor:eq defs of robust}
 effective.
\end{remark}

\section{Relating Character rigidity and stability}\label{sec:char_rigidity}

In this section we  prove Theorem \ref{thm:main_char_rig_INTRO}. 
We recall the properties appearing in Theorem \ref{thm:main_char_rig_INTRO}:
\begin{enumerate}
    \item $\Gamma$ is hyperfinitely Hilbert--Schmidt stable
    \item $\Gamma$ is character rigid.
    \item $\Gamma$  has property~(T;FD)$_\mathrm{rob}$
    \item Every character of $\Gamma$ is a pointwise limit of finite dimensional traces.
\end{enumerate}

\begin{lemma}\label{lem:limits-of-fd implies (T;tr)}
    Let $\Gamma$ be a finitely generated group with property (T;FD). If every trace of $\Gamma$ is a limit of finite dimensional traces, then $\Gamma$ has (T;$\mathrm{W}^*$).
\end{lemma}
\begin{proof}
As in the proof of Proposition    \ref{prop:eq defs of robust}, this follows immediately from Proposition \ref{prop:spec_gap_pass_to_limit}.
\end{proof}

\begin{lemma}\label{lem:T CE implies rigidity}
  Let  $\Gamma$ be a finitely generated group satisfying property  $(T;\mathrm{CE})$. Then any von Neumann amenable character is finite dimensional.
\end{lemma}

\begin{proof}
    Let $\varphi$ be a von Neumann amenable character of $\Gamma$, and let $(M,\tau,\pi)$ be the corresponding tracial representation. Since $M$ is amenable, the representation $\pi\otimes \bar{\pi}$ has almost invariant vectors \cite[Theorem 10.2.9]{Popa}.
    In particular,  $0$ is in the spectrum of $(\pi\otimes\bar\pi)(\Delta_S)$. 

    On the other hand, as $M$ is amenable, it is in particular Connes embeddable. One easily verifies that $M\otimes \bar{M}$, as well as any von Neumann subalgebra of it, is Connes embeddable as well. This in particular applies to the von Neumann algebra $(\pi\otimes\bar{\pi})(\Gamma)''\subset M\otimes \bar{M}$. Thus by (T;$\mathrm{CE}$), $(\pi\otimes \bar{\pi})(\Delta_S)$ has a spectral gap. It follows that $0$ is an eigenvalue of  $(\pi\otimes \bar{\pi})(\Delta_S)$. Now, any corresponding eigenvector is easily seen to be fixed by $\pi\otimes\bar{\pi}(s)$ for each $s\in S$. Since $S$ generates $\Gamma$, we see that $\pi\otimes\bar{\pi}$ has an invariant vector. It follows that $\pi$ contains a finite dimensional subrepresentation \cite[Proposition A.1.12]{Bekka}.  Since $\varphi$ is a character, $\pi$ is factorial, so it must be that $\pi$ is finite dimensional. Hence $\varphi$ is finite dimensional. 
\end{proof}

\begin{proposition}\label{prop:criterion-for-main-thm}
    Let $\Gamma$ be a finitely generated group such that:
    \begin{itemize}
        \item $\Gamma$ has property (T;FD).
        \item $\Gamma$ is residually finite\footnote{This condition can be replaced with $\Gamma$ being maximally almost periodic.}
        \item Every character of $\Gamma$ is either von Neumann amenable, or it is the Dirac trace $\delta_e$. 
    \end{itemize}
   Then the four conditions in Theorem \ref{thm:main_char_rig_INTRO} are equivalent for $\Gamma$.
\end{proposition}
    
\begin{proof}
    By Theorem \ref{thm:HS_for_non_amenable_grps}, $(1)$ is equivalent to
    
    $(1')$: any von Neumann amenable character is a limit of finite dimensional traces. 
    
    Moreover, it is not hard to see using Krein-Milman that  $(1')$ (resp. $(4)$) is equivalent to  every von Neumann amenable trace (resp. every trace) being a limit of finite dimensional traces. 
    
    In addition, by Proposition \ref{cor:eq defs of robust}, $(3)$ is equivalent to 
    
    $(3')$: $\Gamma$ satisfies (T;$\mathrm{CE}$).
    
    We will show that $(1')\Rightarrow(4)\Rightarrow(3')\Rightarrow(2)\Rightarrow(1')$.

    $(1')\Rightarrow (4)$: Let $H_n$ be a sequence of finite index normal subgroups of $\Gamma$ with trivial intersection. The traces corresponding to the quasiregular representations $\pi_n:\Gamma\to U(\ell^2(\Gamma/H_n))$ are the characteristic functions $1_{H_n}$, which clearly converge to $\delta_e$. Any character of $\Gamma$ other than $\delta_e$ is assumed to be von Neumann amenable, and so by $(1')$, it is also a limit of finite dimensional traces. 
    
    $(4)\Rightarrow (3')$: By Lemma \ref{lem:limits-of-fd implies (T;tr)}, $\Gamma$ has property (T;$\mathrm{W}^*$), in particular, it has property (T;$\mathrm{CE}$). 

    $(3')\Rightarrow (2)$: Any von Neumann amenable character of $\Gamma$ is finite dimensional by Lemma \ref{lem:T CE implies rigidity}, hence $\Gamma$ is character-rigid. 

    $(2)\Rightarrow(1')$: This is immediate. 
\end{proof}

Among the three assumptions in Proposition \ref{prop:criterion-for-main-thm} , it is arguably the third one which is most special. 
This condition holds for higher rank S-arithmetic groups in arbitrary characteristic, by the comprehensive work \cite{BBHP,BBH}. Specializing to lattices in semisimple Lie group,  Theorem \ref{thm:main_char_rig_INTRO} from the introduction follows at last.

\begin{proof}[Proof of Theorem \ref{thm:main_char_rig_INTRO}]
We first note that  that $\Gamma$ is finitely generated \cite[Remark 6.18]{raghunathan1972discrete} and thus residually finite by Mal'cev's theorem.
 By  Margulis's arithmeticity theorem, there exists a number field $F$, a semisimple $F$-algebraic group $\bH$, and a continuous surjective homomorphism $p:\bH(\R)\to G$ with compact kernel,  such that $\Gamma$ is commensurable with $p(\bG(\mathcal{O}_F))$ \cite[Chapter XI, Theorem (6.5)]{margulis1991discrete}. By \cite[Theorem B]{BBH},  any character of $\bG(\mathcal{O}_F)$ is either von Neumann amenable or supported on the amenable radical of $\bG(\mathcal{O}_F)$.   

Now, any character $\varphi$ of $\Gamma$ naturally lifts to a character $\tilde{\varphi}$ of $\bG(\mathcal{O}_F)$. It is straightforward to check that  $\tilde{\varphi}$ is von Neumann amenable (resp. supported on the amenable radical) if and only if $\varphi$ satisfies that same respective property. But, as $G$ is center-free and without compact factors, it follows that the amenable radical of $\Gamma$ is trivial  \cite[Lemma 4.8]{BV}. The four conditions are therefore equivalent  for $\Gamma$, by Proposition \ref{prop:criterion-for-main-thm}
\end{proof}

\begin{remark}
    An analogous statement to Theorem \ref{thm:main_char_rig_INTRO} can be made for Invariant random subgroup rigidity and stability with respect to permutations and proven similarly. 
    We choose to not include it in the paper for brevity. 
\end{remark}

\bibliographystyle{alpha}
\bibliography{ref}

\vspace{0.5cm}

\noindent{\textsc{Department of Mathematics, Weizmann Institute of Science, Israel}}

\noindent{\textit{Email address:} \texttt{alon.dogon@weizmann.ac.il}} \\

\noindent{\textsc{Department of Mathematics, University of California San Diego, USA}}

\noindent{\textit{Email address:} \texttt{ivigdorovich@ucsd.edu}} \\

\end{document}